\documentclass{amsart}
\usepackage{enumerate}
\usepackage{textcomp}
\usepackage[arrow, matrix, curve, graph]{xy}
\usepackage{amsfonts}
\usepackage{amsmath}
\usepackage{amssymb}
\usepackage{textcomp}
\usepackage{longtable}
\usepackage{caption}
\usepackage{enumerate}
\usepackage{textcomp}
\usepackage[hyphens]{url}  %% be sure to specify the option 'hyphens'
\makeatletter
\def\@makechapterhead#1{%
  \vspace*{50\p@}%
  {\parindent \z@ \raggedright \normalfont
    \interlinepenalty\@M
    \Huge\bfseries  \thechapter.\quad #1\par\nobreak
    \vskip 40\p@
  }}
\makeatother
\usepackage{tikz}
\usetikzlibrary{matrix,arrows,decorations.pathmorphing}
\usetikzlibrary{decorations.markings}

\newtheorem{theorem}{Theorem}[section]

\newtheorem{lemma}[theorem]{Lemma}

\newtheorem{notation}[theorem]{Notation}
\newtheorem{corollary}[theorem]{Corollary}

\theoremstyle{definition}

\newtheorem{definition}[theorem]{Definition}

\newtheorem{problem}[theorem]{Problem}

\newtheorem{question}[theorem]{Question}

\newtheorem *{Theorem A}{Theorem A}
\newtheorem *{Corollary B}{Corollary B}
\newtheorem *{Theorem C}{Theorem C}
\newtheorem *{Theorem D}{Theorem D}
\newtheorem *{Remark E}{Remark E}
\newtheorem *{Theorem F}{Theorem F}
\newtheorem *{Theorem G}{Theorem G}
\newtheorem *{Theorem H}{Theorem H}

\newcommand{\Gr}{\text{Grad}}

\newcommand{\End}{{\text{End}}}
\newcommand{\Irr}{{\text{Irr}}}
\newcommand{\C}{{\mathbb C}}

\newcommand{\N}{{\mathbb N}}
\newcommand{\RNum}[1]{\uppercase\expandafter{\romannumeral #1\relax}}
\newcommand{\rNum}[1]{\lowercase\expandafter{\romannumeral #1\relax}}
\numberwithin{equation}{section}
 \newcommand{\bigslant}[2]{{\raisebox{.1em}{$#1$}\left/\raisebox{-.1em}{$#2$}\right.}}

\begin{document}
\title[Quotient gradings and the intrinsic fundamental group]{Quotient gradings and the intrinsic fundamental group}
\author{Yuval Ginosar}
\address{Department of Mathematics, University of Haifa, Haifa 3498838, Israel}
\email{ginosar@math.haifa.ac.il}

\author{Ofir Schnabel}
\address{Department of Mathematics, ORT Braude College, 2161002 Karmiel, Israel}
\email{ofirsch@braude.ac.il}

\begin{abstract}

Quotient grading classes are essential participants in the computation of the intrinsic fundamental group $\pi_1(\mathcal{A})$ of an algebra $\mathcal{A}$.
In order to study quotient gradings of a finite-dimensional semisimple complex algebra $\mathcal{A}$
it is sufficient to understand the quotient gradings of twisted group algebra gradings. We establish the graded structure of such quotients using Mackey's obstruction class.
Then, for matrix algebras $\mathcal{A}=M_n(\C)$ we tie up the concepts of \textit{braces,} group-theoretic \textit{Lagrangians} and \textit{elementary crossed products.}
We also manage to compute the intrinsic fundamental group of the diagonal algebras $\mathcal{A}=\C ^4$ and $\mathcal{A}=\C ^5$.
\end{abstract}

\maketitle

\tableofcontents
\bibliographystyle{abbrv}
\section{Introduction}\label{intro}\pagenumbering{arabic} \setcounter{page}{1}
Since it was firstly suggested by
 C. Cibils, M.J. Redondo and A. Solotar in \cite{cibils2010,cibilsintrinsic}, the intrinsic fundamental group
$\pi_1(\mathcal{C})$ was investigated for certain linear categories $\mathcal{C}$
\cite{cibils2010,cibils2012universal,Cibils2016,GS16}. The
intrinsic fundamental group of an associative algebra (as a one-object
category) $\mathcal{A}$ is essentially the inverse limit of a
diagram whose objects are groups which grade $\mathcal{A}$ in a
connected way, and whose morphisms are group epimorphisms which
correspond to quotient morphisms between these gradings. There are
three steps one should in principle do in order to compute
$\pi_1(\mathcal{A})$. First, classify all the connected gradings
of $\mathcal{A}$ up to equivalence of gradings (see below). Then,
compute the quotient grading morphisms between the above grading
classes. These two steps yield a diagram $\Delta (\mathcal{A})$ of
groups and homomorphisms (see \cite[\S 2.4]{GS16} and Definition
\ref{scdiagdef} herein). The last step is to calculate the inverse
limit of this diagram. This is often the hardest of the above
three challenges.

Our concern in this paper are the fundamental groups of finite-dimensional semisimple associative algebras $\mathcal{A}$.
In this regard, to-date $\pi_1(\mathcal{A})$ is fully in hand only
for the diagonal algebras $\mathcal{A}=\oplus_1^j \C,$ where
$j\leq 4$ (\cite{cibils2010}, see also \S\ref{proofofF}), and for
the matrix algebras $\mathcal{A}=M_p(\mathbb{C}),$ where $p$ is
any prime (\cite{cibils2010}, see also Remark E herein).

Let us recall some basic concepts and results about graded (semisimple) algebras. A more thorough background is outlined in \S\ref{review}. Throughout,
all algebras and vector spaces are complex and finite-dimensional, even though these
conditions are sometimes unnecessary.

A \textit{grading} of a linear space $\mathcal{A}$ by a
group $\Gamma$ is just a vector space decomposition
\begin{equation}\label{eq:algebragrading}
\mathcal{G}_{\mathcal{A}}:\mathcal{A}=\bigoplus _{\gamma\in \Gamma}\mathcal{A}_{\gamma}.
\end{equation}
The subspace $\mathcal{A}_{\gamma}\subseteq\mathcal{A}$ is called the \textit{homogeneous $\gamma$-component} of \eqref{eq:algebragrading}.
A $\Gamma$-grading \eqref{eq:algebragrading} is \textit{connected} if
the set
\begin{equation}\label{supp}
\text{supp}_{\Gamma}(\mathcal{A}):=\left\{\gamma\in \Gamma|\ \text{dim}_{\C}(\mathcal{A}_{\gamma})\geq 1\right\}
\end{equation}
generates $\Gamma$.

A $\Gamma$-grading of an algebra $\mathcal{A}$ is a decomposition \eqref{eq:algebragrading} satisfying the extra demand
$\mathcal{A}_{\gamma_1}\cdot\mathcal{A}_{\gamma_2}\subseteq \mathcal{A}_{\gamma_1\cdot\gamma_2}$
for every $\gamma_1,\gamma_2\in \Gamma$.
Due to this inclusion, the homogeneous $e$-component $\mathcal{A}_e$ of \eqref{eq:algebragrading}, where $e$ is the identity element of $\Gamma$, is an $\mathcal{A}$-subalgebra, called the
\textit{base algebra} of $\mathcal{A}$.

A \textit{graded homomorphism} between two graded algebras
\begin{equation}\label{eq:equivgr}
\mathcal{G}_{\mathcal{A}}:\mathcal{A}=\bigoplus _{\gamma\in \Gamma _1}
\mathcal{A}_{\gamma},\quad
\mathcal{G}_{\mathcal{B}}:\mathcal{B}=\bigoplus _{\gamma'\in \Gamma _2}
\mathcal{B}_{\gamma'},
\end{equation}
is a pair $(\psi,\phi)$, where $\psi:\mathcal{A}\rightarrow \mathcal{B}$ is an algebra homomorphism and $\phi:\Gamma _1\rightarrow \Gamma _2$
is a group homomorphism such that
$$\psi(\mathcal{A}_{\gamma})\subseteq \mathcal{B}_{\phi(\gamma)},\ \ \forall \gamma\in\Gamma.$$
If $\Gamma_1=\Gamma_2(=\Gamma)$ and $\phi:\Gamma\rightarrow \Gamma $ is the identity we say that $\psi$ is a $\Gamma$-graded homomorphism and write
$$\psi:\mathcal{A}\stackrel{\Gamma}{\to}\mathcal{B}.$$
A graded homomorphism  $(\psi,\phi)$ is a \textit{graded quotient morphism} if $\psi$ is an algebra isomorphism and the group homomorphism $\phi$ is surjective.
A quotient-morphism $(\psi,\phi)$ between two graded algebras
\eqref{eq:equivgr} is a {\it graded-equivalence} if $\phi$ is a
group isomorphism. We denote the graded equivalence class
of~\eqref{eq:algebragrading} by $[\mathcal{G}_{\mathcal{A}}]$.
With the above notation, a graded equivalence $(\psi,\phi)$ is a {\it graded isomorphism}
if $\Gamma_1=\Gamma_2$ and $\phi:\Gamma_1\rightarrow \Gamma _1$ is the identity. We denote this finer equivalence relation of $\Gamma$-graded isomorphism by $\stackrel{\Gamma}{\cong}$.

Quotients of graded algebras can be given just in terms of a quotient of
the grading group as follows. Let \eqref{eq:algebragrading} be a $\Gamma$-graded
algebra and let $N\lhd \Gamma$ be a normal subgroup. Then the quotient ${\Gamma}/{N}$-grading of \eqref{eq:algebragrading} by $(\psi,\phi)=(\text{Id}_\mathcal{A},\text{mod}N)$,
is denoted by $\bigslant{\mathcal{G}_{\mathcal{A}}}{N}$,
that is
$$\bigslant{\mathcal{G}_{\mathcal{A}}}{N}:\mathcal{A}=\bigoplus _{\bar{\gamma}\in \Gamma/N} \mathcal{A}_{\bar{\gamma}},$$
where $\mathcal{A}_{\bar{\gamma}}:=\bigoplus_{\gamma\in\bar{\gamma}}\mathcal{A}_{\gamma}.$

Quotient morphisms determine a well-defined partial order on the
set Grad$(\mathcal{A})$ of connected graded-equivalence classes of
a given finite-dimensional algebra $\mathcal{A}$ \cite[Proposition
2.8]{GS16}. If there exists a quotient morphism
$(\psi,\phi):\mathcal{G}_1\to\mathcal{G}_2$ between two gradings
of $\mathcal{A}$ then we say that $[\mathcal{G}_2]$ is a
\textit{graded quotient} of $[\mathcal{G}_1],$ and write
$$[\mathcal{G}_2]\leq [\mathcal{G}_1].$$
Note that the graded class of the algebra $\mathcal{A}$, considered as graded by the trivial group $\{e\}$, is the minimum element in this partial order.
For a connected grading class $[\mathcal{G}]\in$Grad$(\mathcal{A})$, let
\begin{equation}\label{eq:conus}
[\mathcal{G}]_{\geq}:=\{\mathcal{H}\in
\text{Grad}(\mathcal{A})|\ [\mathcal{H}]\leq [\mathcal{G}]\}
\end{equation}
be the cone corresponding to $[\mathcal{G}]$.
With the notation \eqref{eq:conus}, a grading class of $\mathcal{A}$ is a \textit{common quotient} of the grading classes
$[\mathcal{G}_1],[\mathcal{G}_2]\in\text{Grad}(\mathcal{A})$ if it belongs to the intersection
\begin{equation}\label{intsect}
[\mathcal{G}_1]_{\geq}\bigcap [\mathcal{G}_2]_{\geq}\subseteq\text{Grad}(\mathcal{A}).
\end{equation}
Then, for the computation of $\pi_1(\mathcal{A})$ in the case where $\mathcal{A}$ is a finite-dimensional algebra, it is
enough to classify all the maximal connected grading classes of
$\mathcal{A}$ under the quotient partial order, as well as their maximal common quotients,
and then to compute the pull-back
of the corresponding sub-diagram of $\Delta (\mathcal{A})$. Our objective is to advance in the study of the above tasks in the semisimple case.
For this purpose, Problems \ref{prob1} and \ref{q:maxelemunique} and Question \ref{questiona} are hereby posed.

Let us recall two important ingredients in the theory of graded algebras.\\
\textbf{Induced gradings.} (see \cite[\S 2.5]{GS16} and \S\ref{inductionsec} in the sequel). A $\Gamma$-grading
\begin{equation}\label{eq:vsgrading}
V=\bigoplus _{\gamma\in \Gamma}V_{\gamma}.
\end{equation}
of an $r$-dimensional vector space $V$ yields a natural $\Gamma$-grading on the algebra of $r$-by-$r$ complex matrices identified with its endomorphism algebra as follows:
\begin{equation}\label{eq:elemgrading}
\text{M}_r(\C)\cong\End_{\C}(V)=\bigoplus _{\gamma\in \Gamma}\End_{\C}(V)_{\gamma},
\end{equation}
where
\begin{equation}\label{endintro}
\End_{\C}(V)_{\gamma}:=\left\{\varphi\in\End_{\C}(V)|\ \ \varphi(V_{\gamma'})\subseteq V_{\gamma\cdot\gamma'},\ \ \forall \gamma'\in \Gamma\right\}.
\end{equation}
The grading \eqref{eq:elemgrading} is termed an \textit{elementary grading} \cite[P. 711]{das1999} of the matrix algebra (associated with the vector space grading \eqref{eq:vsgrading}).
The grading \eqref{eq:vsgrading}, and in turn its associated elementary grading \eqref{eq:elemgrading} give rise to a character
\begin{equation}\label{char}
x:=\sum_{\gamma\in\Gamma}\dim_{\C}(V_{\gamma})\cdot\gamma\in\N [\Gamma]
\end{equation}
of augmentation $\epsilon(x)=r$, where $\N[\Gamma]$ denotes the group semiring of $\Gamma$ over $\N$ (we assume $0\in \N$).
Given a
grading~\eqref{eq:algebragrading} of an algebra $\mathcal{A}$ and
a graded $r$-dimensional space \eqref{eq:vsgrading}, there is a natural way to grade
the matrix algebra $M_r(\mathcal{A})$ identifying it with the tensor product $\End_{\C}(V)\otimes_{\C}\mathcal{A}$ (see e.g. \cite{EK13,NVO82} and also \cite[\S 2.5]{GS16}).
We say that this grading of $M_r(\mathcal{A})$ is \textit{induced} from the
grading~\eqref{eq:algebragrading} by the corresponding element $x\in \mathbb{N}[\Gamma]$ and denote this grading class by $[x(\mathcal{G}_{\mathcal{A}})]$.
In particular, a grading class of $M_r(\C)$ is elementary if and only if it is
induced from the trivial grading of $\C$ (see e.g. \cite{EK13,NVO82} and also \cite[Definition 2.12]{GS16}). An elementary grading class of $M_r(\C)$ which corresponds to $x\in \mathbb{N}[\Gamma]$
with $\epsilon(x)=r$ is denoted by $[x(\C)]$.\\
\textbf{Twisted group algebras.} Let $G$ be a group. A \textit{twisted group algebra} over $G$ is
a $|G|$-dimensional $\C$-algebra
\begin{equation}\label{tga}
\C^{\alpha}G=\bigoplus_{g\in G}\C u_g=\bigoplus_{g\in G} u_g\C
\end{equation}
endowed with multiplication, which is defined by
a $G$-indexed basis $\{u_{g}\}_{g\in G}$ of invertible homogeneous elements
$$u_{g_1}u_{g_2}=\alpha(g_1,g_2)u_{g_1g_2},\ \ g_1,g_2\in G,$$
where $\alpha\in Z^2(G,\C^*)$, that is a 2-cocycle over $\C^*$ with the trivial $G$-action.
Any twisted group algebra $\C^{\alpha}G$ admits a natural $G$-grading by letting
\begin{equation}\label{eq:natgradtga}
(\C^{\alpha}G)_g:=\text{span}_{\C}\{u_g\},\ \ \forall g\in G.
\end{equation}
We call these gradings \textit{twisted group algebra} [TGA] gradings. It should be pointed out that the TGA gradings $\C^{\alpha_1}G$ and $\C^{\alpha_2}G$ are graded isomorphic if and only if the 2-cocycles
$\alpha_1$ and $\alpha_2$ are cohomologically-equivalent (see e.g. \cite{EK13} and also \cite[Proposition 2.4(1)]{GS16}).
Graded-equivalence of these
TGA gradings is given in \cite[Proposition 2.4(2)]{GS16} (see also \cite{EK13}) in terms of Aut$(G)$-orbits in $H^2(G,\C^*)$. TGA grading classes by finite groups are maximal with
respect to the partial order $``\leq"$ on connected gradings of finite-dimensional algebras (see \cite[Proposition 2.31]{EK13}).

Inductions of TGA gradings, described in the above subsections, are \textit{simply-graded}, i.e. they admit no proper graded ideals (see Definition~\ref{def:simplygraded}).
Furthermore, by a generalized Maschke's theorem \cite[Theorem 4.4]{P89}, complex twisted group algebras of finite groups are semisimple,
and in turn, so are direct sums of inductions of such twisted group algebras.
Conversely, any $\Gamma$-grading of a complex semisimple algebra $\mathcal{A}$ is obtained from TGA gradings by direct sums and inductions \cite{BSZ,MR1941224,MR2488221} (see also \S\ref{gsa}).
In other words, there exist elements $x_i\in \N [\Gamma]$, subgroups $G_i<\Gamma$ and 2-cocycles $\alpha_i\in Z^2(G_i,\C^*)$
such that $[x_i(\C^{\alpha_i}G_i)]$ are the simply-graded summands of $[\mathcal{G}_{\mathcal{A}}]$, that is
\begin{equation}\label{generalsum}
[\mathcal{G}_{\mathcal{A}}]=\left[\bigoplus_{i=1}^lx_i(\C^{\alpha_i}G_i)\right].
\end{equation}

Direct summations, as well as inductions, commute with quotient
morphisms as follows. For any $x=\sum n_{\gamma} \gamma \in \N [\Gamma]$ and any normal subgroup $N\lhd \Gamma$ write $\bar{x}:=\sum n_{\gamma} \gamma N\in \N [\Gamma/N].$ Then
\begin{equation}\label{reduct}
\left[\bigslant{\mathcal{G}_{\mathcal{A}}}{N}\right]=\left[\bigoplus_{i=1}^l\overline{x_i}\left(\bigslant{\C^{\alpha_i}G_i}{N\cap G_i}\right)\right],
\end{equation}
where $\bigslant{\C^{\alpha_i}G_i}{N\cap G_i}$ is the quotient grading of the TGA grading $\C^{\alpha_i}G_i$ by the normal subgroup $N\cap G_i$ of $G_i$ for every $i=1,\cdots, l$.
Hence, the problem of finding the quotient gradings of general
gradings of complex semisimple algebras can be reduced to
\begin{problem}\label{prob1}
Let $\alpha\in Z^2(G,\C^*)$ be a 2-cocycle of a group $G$ and let $N\lhd G$ be a normal subgroup.
Find elements $y_i\in \N [G/N]$, subgroups $H_i<G/N$ and 2-cocycles $\alpha_i\in Z^2(H_i,\C^*)$ such that
$\left[\bigslant{\C^{\alpha}G}{N}\right]=\left[\bigoplus_{i=1}^ly_i(\C^{\alpha_i}H_i)\right]$.
\end{problem}
We remark that a partial answer to Problem 1.1 for $G$ abelian can be deduced from \cite[Lemma 3]{EK15}.

The twisted group algebra $\C^{\alpha}G$ determines a $G/N$-action \eqref{star} on the set Irr$(\C ^{\alpha}N)$
of isomorphism types of irreducible $\C ^{\alpha}N$-modules (alternatively,
the set Irr$(N,\alpha)$ of irreducible $\alpha$-projective representations of $N$), where we keep denoting the restriction of $\alpha \in Z^2(G,\C ^*)$ to $Z^2(N,\C ^*)$ again by $\alpha$.

For $[M]\in$Irr$(\C ^{\alpha}N)$ let $\mathcal{I}_{M}=\mathcal{I}_{\C^{\alpha}G}(M)<G/N$ be its
stabilizer subgroup (or the {\it inertia} subgroup) under
the $G/N$-action \eqref{star}, and let $T_{M}$ be a transversal set of $\mathcal{I}_{M}$ in $G/N$.
Then the orbit of $[M]$ in Irr$(\C ^{\alpha}N)$ under the $G/N$-action is given by $\{[t(M)]\}_{t\in T_{M}}$.
The following answer to Problem \ref{prob1} is Clifford's Theory in terms of quotient gradings,
it makes use of the notion of Mackey's obstruction cohomology class \cite[Theorem 8.3]{M58} over the inertia group (see Theorem~\ref{let}).
\begin{Theorem A}\label{inertia}
    Let $G$ be a finite group, $[\alpha]\in H^2(G,\C^*)$, let $N\lhd G$ a normal subgroup and let $[\mathcal{G}]$ be the quotient $G/N$-grading class of the twisted
    grading class $[\C^{\alpha}G]$. Then with the above notation
    \begin{enumerate}
\item There is a one-to-one correspondence between the simply-graded summands of $[\mathcal{G}]$  and the orbits in Irr$(\C ^{\alpha}N)$ under the $G/N$-action \eqref{star}.
        \item Let $[\mathcal{G}_{M}]$ be a simply-graded summand of $[\mathcal{G}]$ which corresponds to the  $G/N$-orbit of $[M]\in$Irr$(\C ^{\alpha}N)$.
Then $$[\mathcal{G}_{M}]=\left[x_{M}(\C^{\omega_{}}\mathcal{I}_{M})\right],$$ where $x_{M}=\text{dim}_{\C}(M)\cdot\sum _{t\in T_{M}}t\in \mathbb{N}[G/N]$ and
        $[\omega_{}]=\omega_{\mathcal{G}}([M])\in H^2(\mathcal{I}_{M},\C^*)$
        is Mackey's obstruction cohomology class which corresponds to $[M]$ with respect to the action \eqref{star}.
    \end{enumerate}
\end{Theorem A}
The proof of Theorem A involves the property of TGA grading classes $[\C^{\alpha}G]$, as well as their quotient gradings, of admitting the same dimension at every homogeneous component.
In \S\ref{equi} we discuss this equi-dimensional feature of graded algebras.

As mentioned above, complex twisted group algebras over finite groups are semisimple. In certain cases they may even be simple.
If such a complex twisted group algebra $\C^{\alpha}G$ is simple,
i.e. isomorphic to a matrix algebra $M_n(\mathbb{C})$, then the
$2$-cocycle $\alpha\in Z^2(G,\C^*)$ is called {\it
non-degenerate}. Such cocycles can be considered as
group-theoretical analogues of symplectic forms on linear spaces
\cite{david2013isotropy}. Non-degeneracy is a cohomology class
property. We thus refer also to non-degenerate cohomology classes.
A group admitting a non-degenerate cocycle is called {\it of
central type} (CT) \footnote{In \cite{Gagola} and in \cite{ShSh}
such groups are termed {\it fully ramified} and {\it central type
factor groups} respectively (see \cite[\S 2.7]{GS16}).} and is
evidently of square order. The simplicity condition says that a
cohomology class $[\alpha]\in H^2(G,\C^*)$ is non-degenerate if
and only if $\Irr(\C ^{\alpha}G)$ is a singleton. Obviously, in this case the
unique element $[M]\in\Irr(\C ^{\alpha}G)$ satisfies
    $\dim_{\C}(M)=\sqrt{|G|}$ (see \cite[\S 1]{ShSh}).
The following is a consequence of Mackey's correspondence that can also be directly derived from Theorem A (see also Corollary \ref{elqu}).
\begin{Corollary B}
Let $G$ be a group of CT.    If both $[\alpha]\in H^2(G,\C^*)$ and its restriction to a normal subgroup $N\lhd G$ are non-degenerate,
    then the quotient group $G/N$ is of CT.
\end{Corollary B}

Corollary B brings us to matrix algebras. Let (Grad$(M_n(\C)),\leq)$ be the poset of all connected grading classes of ${\mathcal{A}}=M_n(\C)$.
By \eqref{generalsum} (see also Theorem \ref{BSZ}), any grading class
$[\mathcal{G}_{M_n(\C)}]\in$Grad$(M_n(\C))$ is induced from a
TGA grading class $[\C ^{\alpha}G]$ for some non-degenerate $\alpha \in Z^2(G,\C^*)$ of a CT group $G$, whose order divides $n^2$.

\begin{notation}\label{not}
Let $G$ be a group of CT whose order divides $n^2$, with $\alpha \in Z^2(G,\C^*)$ non-degenerate. Denote by
\begin{equation}\label{poset}
\Gr(M_n(\C))_{\C ^{\alpha}G}\subseteq\Gr(M_n(\C))
\end{equation}
the set of grading classes which are induced from a TGA class of $\C ^{\alpha}G$ as above.
Further, for $[\mathcal{G}]\in$Grad$(M_n(\C))$ we denote
\begin{equation}\label{eq:allquotient}
[\mathcal{G}]_{\C ^{\alpha}G}:=\text{Grad}(M_n(\C))_{\C ^{\alpha}G}\cap[\mathcal{G}]_{\geq}(=\{[\mathcal{G}']\in\text{Grad}(M_n(\C))_{\C ^{\alpha}G}|\ [\mathcal{G}']\leq[\mathcal{G}]\}).
\end{equation}
In particular $[\mathcal{G}]_{\C }$ is the set of all elementary quotient classes of $[\mathcal{G}]$.
\end{notation}
It should be remarked that the set \eqref{eq:allquotient} may be empty.
Next, the partially ordered subset (Grad$(M_n(\C))_{\C ^{\alpha}G},\leq)$ in \eqref{poset}, where $G$ a CT group whose order divides $n^2$ and $\alpha \in Z^2(G,\C^*)$ non-degenerate, admits a maximum
herein explained. Let $$d=d({G})=\frac{n}{\sqrt{|G|}},$$ let
$\mathcal{F}_{d-1}$ be the free group of rank $d-1$ generated by,
say, $\{x_1,\cdots, x_{d-1}\}$ and let
\begin{equation}\label{maxinduceelement}
\tilde{x}^d:=1+\sum_{i=1}^{d-1}x_i\in \N[\mathcal{F}_{d-1}].
\end{equation}
Then the maximum of
Grad$(M_n(\C))_{\C ^{\alpha}G}$ is graded by the free product
$\mathcal{F}_{d-1}*G$ of the free group $\mathcal{F}_{d-1}$ and the CT group $G$, and is of the form
\begin{equation}\label{maxelement}
\max(\Gr(M_n(\C))_{\C ^{\alpha}G})=[\tilde{x}^d(\C ^{\alpha}G)],
\end{equation}
i.e. it is induced from the TGA class $[\C ^{\alpha}G]$ by the element \eqref{maxinduceelement}.
Since any grading class in $\text{Grad}(M_n(\C))$ belongs to some $\Gr(M_n(\C))_{\C ^{\alpha}G}$, then
all the maximal classes of $\text{Grad}(M_n(\C))$ are of the form \eqref{maxelement} (see \cite[Proposition 2.31 and Corollary 2.34]{EK13},\cite[Theorem 1.3]{GS16}).
In particular, there are finitely many maximal classes in $\text{Grad}(M_n(\C))$.

The second step in the procedure of estimating the intrinsic fundamental group of a finite-dimensional algebra $\mathcal{A}$ is to
find the maximal common quotients in the partially-ordered subset \eqref{intsect}, running over the maximal classes $[\mathcal{G}_1],[\mathcal{G}_2]\in\text{Grad}(\mathcal{A})$.
For $\mathcal{A}=M_n(\C)$ this task boils down to studying maximal classes in the intersection
\begin{equation}\label{comquot matalg}
\left[\tilde{x}^{d_1}(\C^{\alpha_1}G_1)\right]_{\geq}\ \bigcap\ \left[\tilde{x}^{d_2}(\C^{\alpha_2}G_2)\right]_{\geq}
\end{equation}
for all maximal grading classes \eqref{maxelement} of $\mathcal{A}$.
Let us focus on the case where $d_1=1$ and $G_2=\{e\}$ (or $d_2=n$) in \eqref{comquot matalg}, that is the elementary quotients of TGA grading classes.
Under this setup, the problem is rephrased as
\begin{problem}\label{q:maxelemunique}
Let $G$ be a CT group and $[\alpha]\in H^2(G,\mathbb{C}^*)$ non-degenerate.
Describe the maximal classes in $[\C^{\alpha}G]_{\C}$ (see \eqref{eq:allquotient}).
In particular, determine when $[\C^{\alpha}G]_{\C}$ admits a \textit{unique} maximal class.
\end{problem}
For $|\Gamma|=n$, there exists a special elementary $\Gamma$-grading of $M_n(\C)$, namely an {\it elementary crossed product $\Gamma$-grading}
\footnote{In \cite[Definition 3]{ak13} these gradings are just called ``crossed-products".
Our notion of crossed-products is more common, see \S \ref{CPSec}.}.
This is the graded endomorphism algebra of the group algebra $\C \Gamma$ as a free module over itself.
 Such a $\Gamma$-grading is induced from the trivial grading on $\C$ by $\sum _{\gamma\in \Gamma} \gamma\in\mathbb{N}[\Gamma]$ and is therefore denoted by $\sum _{\gamma\in \Gamma} \gamma(\C)$.

TGA grading classes $[\C ^{\alpha}G]$ of complex matrix algebras do not necessarily admit an elementary crossed product quotient grading class.
However, if a TGA grading class $[\C ^{\alpha}G]$ of a complex matrix algebra does admit an elementary crossed product quotient grading class,
then this elementary crossed product class is maximal (not necessarily unique) among the elementary quotient grading classes of $[\C ^{\alpha}G]$ (see Corollary~\ref{cor:CPmax}). We therefore ask
\begin{question}\label{questiona}
Let $G$ be a CT group, and let $[\alpha]\in H^2(G,\mathbb{C}^*)$ be non-degenerate. For which normal subgroups $N\lhd G$ (if at all)
is the quotient $[\bigslant{\C^{\alpha}G}{N}]$ an elementary crossed product grading class?
\end{question}
An answer to Question \ref{questiona} involves the group-theoretical concept of a Lagrangian subgroup hereby explained. Let $[\alpha]\in H^2(G,\C^*)$ be a cohomology class of a finite group $G$.
A subgroup $H<G$ is called {\it $[\alpha]$-isotropic} if $[\alpha]$ is trivial when restricted to $H$.
If, additionally, $[\alpha]$ is non-degenerate and $|H|^2=|G|$, we say that $H$ is a {\it Lagrangian} with respect to $[\alpha]$.
The following theorem answers Question \ref{questiona}.
\begin{Theorem C}\label{tgacp}
A quotient $[\bigslant{\C^{\alpha}G}{N}]$ of a TGA grading class $[\C^{\alpha}G]$ of a simple complex algebra is an elementary crossed product grading class
if and only if the normal subgroup $N\lhd G$ is an abelian Lagrangian with respect to the non-degenerate cohomology class $[\alpha]\in H^2(G,\C^*)$.
\end{Theorem C}
As opposed to symplectic linear forms, non-degenerate cohomology classes do not necessarily admit normal Lagrangians as shown in \cite[\S 1]{david2013isotropy}.
Thus, by Theorem C there exist maximal elementary quotients of TGA grading classes which are not elementary crossed product gradings.
However, if a CT group $G$ is either abelian or nilpotent of order which is free of eighth powers then any non-degenerate $[\alpha]\in H^2(G,\C^*)$ admits a normal Lagrangian
\cite[Theorem 1.9]{david2013isotropy}.
By Theorem C, the corresponding simple twisted group algebras over these two families of CT groups admit elementary crossed product quotient classes, answering Question \ref{questiona}.
For the first one of these families, namely of the abelian CT groups, the following result suggests a solution to Problem \ref{q:maxelemunique} in terms of elementary crossed product grading classes.
\begin{Theorem D}\label{th:Theorem A}
Let $A$ be an abelian group of CT whose order is $n^2$ and let $[\alpha]\in H^2(A,\C ^*)$ be non-degenerate.
\begin{enumerate}
\item  An elementary quotient grading class in $\text{Grad}(M_n(\C))$ is maximal in $[\C ^{\alpha} A]_{\C}$ if and only if it is an elementary crossed product grading class.
 \item There is a unique maximal class in $\left[\C ^{\alpha} A\right]_{\C}\subseteq\text{Grad}(M_{n}(\C))$
 if and only if $A$ is elementary abelian.
\end{enumerate}
\end{Theorem D}
For the following remark, recall that (see \S\ref{abCT}) for any $n\in \N$, the group $A=C_n\times C_n$ admits a non-degenerate cohomology class $[\alpha]\in H^2(A,\C ^*)$.
\begin{Remark E}
In \cite[Proposition 4.15]{cibils2010} it is claimed that for every $n$ the grading classes
$[1(\C^{\alpha} C_n\times C_n)]$ and $[\tilde{x}^n(\C)]$ of $M_n(\mathbb{C})$
possess a \textit{unique} maximal common quotient grading class (graded by the cyclic group $C_n$), namely the elementary crossed product grading class $\left[\left( \sum _{h\in C_n}h\right)(\C)\right]$.
However, the proof of maximality of this class is not
precise. Moreover, by Theorem D for the abelian group $A:=C_n\times C_n$, the uniqueness part is true if and only if $n$ is square-free.
Nevertheless, the computation of $\pi _1 (M_p(\mathbb{C}))$ in \cite{cibils2010} is unaffected for primes $p$.
\end{Remark E}

It is natural to pose a converse to Question \ref{questiona}.
\begin{question}\label{questionb}
Given a group $H$ of order $n$,
does there exist a CT group $G$ (of order $n^2$) and a non-degenerate $[\alpha] \in H^2(G,\mathbb{C}^*)$ such that the elementary crossed product grading class $\left[\sum _{h\in H}h(\mathbb{C})\right]$
is a quotient grading of $[\mathbb{C} ^{\alpha}G$]?
\end{question}

A highly non-trivial theorem says that CT groups are solvable \cite{isaacs,LY}.
Consequently, a necessary condition for $H$ to be a quotient of a CT group $G$ is that $H$ itself is solvable.
It turns out that solvability of $H$ is not sufficient for an affirmative answer to Question \ref{questionb}. For this purpose we need the notion of {\it involutive Yang-Baxter} (IYB) groups.
A finite group $H$ is IYB \cite{CJdR} if it admits a module $M$ and a 1-cocycle $\delta\in Z^1(H,M)$ which is bijective.
Identification of an IYB group $H$ and its module $M$ via the bijection $\delta:H\to M$ yields a unified structure, furnished with two compatible operations.
It was firstly introduced in \cite{rump2007braces} and named a {\it brace}. The current extensive study of braces is surveyed in \cite{Cedo}.
By \cite[Theorem 2.15]{ESS} IYB groups are indeed solvable, however in \cite{B16} a certain finite nilpotent group is shown not to be IYB.
The answer to Question \ref{questionb} yields an alternative characterization of braces as follows.
\begin{Theorem F}
A group $H$ of order $n$ is IYB if and only if its corresponding elementary crossed-product grading class $\left[\sum _{h\in H}h(\mathbb{C})\right]$ is a quotient of
some TGA grading class of $M_n(\mathbb{C})$.
\end{Theorem F}
Lastly, the intrinsic fundamental group of the diagonal algebra $\C^i$ is
computed in \cite{cibils2010} for $1\leq i\leq 4$. The case $i=4$
\cite[Theorem 6.9]{cibils2010} turns out to be incorrect. We fix
this error and compute $\pi_1(\C^i)$ for $i=4,5$.
\begin{Theorem G}
The intrinsic fundamental groups of the complex diagonal algebras of ranks 4 and 5 are
\begin{enumerate}
\item   $\pi _1(\C ^4)\cong H_4 \times C_6$, where $H_4$ is the central extension
$$\quad 1\rightarrow \stackrel{\{ \pm 1\}}{C_2}\rightarrow H_4\stackrel{\beta_4}{\rightarrow} \stackrel{\langle a\rangle}{C_2}*\stackrel{\langle b\rangle}{C_2}\rightarrow 1$$
determined by $\bar{a}^2=\bar{b}^2=-1$ for a choice of $\bar{a}\in \beta_4 ^{-1}(a)$ and $\bar{b}\in \beta_4 ^{-1}(b)$.
\vspace*{0.5cm}
\item
    $\pi _1(\C ^5)\cong H_5 \times C_{10},$ where $H_5$ is the central extension
$$\quad 1\rightarrow \stackrel{\{ \pm 1\}}{C_2}\rightarrow H_5\stackrel{\beta_5}{\rightarrow} \stackrel{\langle z\rangle}{C_6}*\stackrel{\langle w\rangle}{C_2}\rightarrow 1$$
determined by $\bar{z}^6=\bar{w}^2=-1$ for a choice of $\bar{z}\in \beta_5 ^{-1}(z)$ and $\bar{w}\in \beta_5 ^{-1}(w)$.
\end{enumerate}
\end{Theorem G}
\section{Preliminaries}\label{review}
We begin this section recording a group-theoretical result that is useful in \S\ref{equi}. For a subgroup $G<\Gamma$
let $[\Gamma:G]=\{\gamma G\}_{\gamma\in \Gamma}$ be the corresponding set of left cosets.
Any $\gamma \in \Gamma$ gives rise to a one-to-one correspondence
\begin{eqnarray}\label{sigamma}\sigma _{\gamma}:
\begin{array}{ccc}
[\Gamma:G]&\rightarrow &[\Gamma:G]\\
\gamma'G&\mapsto &\gamma\cdot\gamma'G
\end{array},
\end{eqnarray}
and hence determines a permutation in $\text{Sym}_{[\Gamma:G]}$. The following result can easily be verified and appears in many textbooks.
\begin{lemma}\label{lemma:permu}
    With the above notation, the map
    \begin{eqnarray}\label{psigamma}
    \begin{array}{ccc}
    \Gamma&\rightarrow &\text{Sym}_{[\Gamma:G]}\\
    \gamma&\mapsto &\sigma _{\gamma}
    \end{array}
    \end{eqnarray}
    is a group homomorphism determining a transitive left action of $\Gamma$ on $[\Gamma:G]$.
\end{lemma}
\subsection{Ablian groups of central type and their non-degenerate classes}\label{abCT}
Let $A$ be a finite abelian group. The second cohomology group $H^2(A,\C^*)$ is easily described as follows.
A function $\chi :A\times A\to \C^*$ is called an \textit{alternating bi-character} (see \cite[P. xx]{EK13}) if for every $\sigma,\tau,\eta\in A$
\begin{enumerate}
\item $\chi(\sigma\tau,\eta)=\chi(\sigma,\eta)\chi(\tau,\eta),$
\item $\chi(\sigma,\tau\eta)=\chi(\sigma,\tau)\chi(\sigma,\eta)$, and
\item $\chi(\sigma,\sigma)=1$.
\end{enumerate}
The set of alternating bi-characters of $A$ is an abelian group itself with the pointwise multiplication.
Then, any 2-cocycle $\alpha\in Z^2(A,\C^*)$ determines an alternating bi-character
\begin{eqnarray}\label{betaf}\chi_{\alpha}:
\begin{array}{rcl}A\times A&\to& \C^*\\
(\sigma,\tau)&\mapsto & \alpha(\sigma,\tau)\cdot \alpha(\tau,\sigma)^{-1}
\end{array},
\end{eqnarray}
such that the rule $[\alpha]\mapsto\chi_{\alpha}$ is a well-defined isomorphism between $H^2(A,\C^*)$ and the group of alternating bi-characters of $A$.
Explicitly, if $\{u_{\sigma}\}_{\sigma\in A}$ is a basis of invertible homogeneous elements of $\C^{\alpha}A$ as in \eqref{tga} then
\begin{equation}\label{commust}
\chi_{\alpha}(\sigma,\tau)=u_{\sigma}\cdot u_{\tau}\cdot u_{\sigma}^{-1}\cdot u_{\tau}^{-1}=[u_{\sigma}, u_{\tau}],
\end{equation}
that is the multiplicative commutator of the invertible elements $u_{\sigma}$ and $u_{\tau}$ for every $\sigma,\tau\in A$.
The following well-known theorem characterizes non-degeneracy in terms of alternating bi-characters.
\begin{theorem}\label{CTab}(see e.g. \cite[\S 3.1]{GS16})
Let $A$ be abelian group. Then $[\alpha]\in H^2(A,\C ^*)$ is non-degenerate if and only if there exists a subgroup (which depends on $[\alpha]$)
\begin{equation}\label{A1dec}
A_1=\langle x_1\rangle \times \langle x_2\rangle \times \ldots \times \langle x_r\rangle\cong C_{n_1}\times C_{n_2}\times \ldots \times C_{n_r}
\end{equation}
of $A$,
and an embedding of abelian groups $\phi:A_1\to A$, such that $A=A_1\times \phi (A_1)$ and whose corresponding alternating bi-character \eqref{betaf} is
$$\chi_{\alpha}:\left(\prod_{i=1}^{r}x_i^{l_i}\cdot\phi(x_i)^{l'_i},\prod_{i=1}^{r}x_i^{m_i}\cdot\phi(x_i)^{m'_i}\right)\mapsto\prod_{i=1}^{r}\zeta _i^{l_i\cdot m'_i-l'_i\cdot m_i},$$
where $\zeta _i$ is a primitive $n_i$-th root of unity for every $1\leq i\leq r$.
\end{theorem}
In particular, for any $n\in \N$ the group $C_n\times C_n$ is of CT (see Remark E).
It is not hard to verify that in this case any choice of a cyclic subgroup $A_1<A$ of order $n$ is good for the decomposition in Theorem \ref{CTab}.
The following is another immediate consequence of Theorem \ref{CTab}.
\begin{corollary}\label{CorCTab}
Let $B\subseteq\{1,\cdots,r\}$ be any subset.
Then with the notation \eqref{A1dec}, the restriction of $[\alpha]$ to $\langle\{ x_i\}_{i\in B}\rangle\times \phi (\langle\{ x_i\}_{i\in B}\rangle)$ is non-degenerate as well.
\end{corollary}
\subsection{Graded modules, graded endomorphisms and induction}\label{inductionsec}
A left (right) module of a $\Gamma$-graded complex algebra \eqref{eq:algebragrading} is {\it graded} if it affords a linear decomposition
    \begin{equation}\label{modec}
    \mathcal{W}_{}=\bigoplus_{\gamma\in \Gamma} \mathcal{W}_{\gamma}
    \end{equation}
    which respects the grading \eqref{eq:algebragrading}, that is
    for every $\gamma_1,\gamma_2\in \Gamma$
    \begin{equation}\label{stgrmod}
\mathcal{A}_{\gamma_1}\mathcal{W}_{\gamma_2}\subseteq \mathcal{W}_{\gamma_1\cdot\gamma_2}
(\text{respectively, }\mathcal{W}_{\gamma_1}\mathcal{A}_{\gamma_2}\subseteq \mathcal{W}_{\gamma_1\cdot\gamma_2}).
    \end{equation}
%Any two-sided free $\mathcal{A}$-module is an immediate example of a left and right graded module over \eqref{eq:algebragrading}.
%\subsection{Graded endomorphisms}\label{gsaend}
Let \eqref{modec} be a graded left module over a graded algebra \eqref{eq:algebragrading}.
A $\C$-endomorphism $\varphi:\mathcal{W}_{}\to\mathcal{W}_{}$ is \textit{left (right) homogeneous of degree $\gamma$} if
$\varphi(\mathcal{W}_{\gamma'})\subseteq\mathcal{W}_{\gamma\cdot\gamma'}$
(respectively, $\varphi(\mathcal{W}_{\gamma'})\subseteq\mathcal{W}_{\gamma'\cdot\gamma}$) for every $\gamma'\in\Gamma$.
The $\C$-space of left (right) homogeneous $\mathcal{W}_{}$-endomorphisms of degree $\gamma\in \Gamma$ is denoted by End$_{\C}^{\textmd{l}(\gamma)}(\mathcal{W}_{})$
(respectively, End$_{\C}^{\textmd{r}(\gamma)}(\mathcal{W}_{})$).
The direct sums of these homogeneous subspaces of End$_{\C}(\mathcal{W}_{})$
are $\Gamma$-graded algebras denoted by
\begin{eqnarray}\label{grend}
\begin{array}{l}\text{End}_{\C}^{\text{l}(\Gamma)}(\mathcal{W}_{}):=\bigoplus_{\gamma\in \Gamma}\text{End}_{\C}^{\textmd{l}(\gamma)}(\mathcal{W}_{}),\text{ and }\\
\text{End}_{\C}^{\text{r}(\Gamma)}(\mathcal{W}_{}):
=\bigoplus_{\gamma\in \Gamma}\text{End}_{\C}^{\textmd{r}(\gamma)}(\mathcal{W}_{}).
\end{array}
\end{eqnarray}

The algebra of right graded endomorphisms $\text{End}_{\C}^{\text{r}(\Gamma)}(\mathcal{W}_{})$ contains an important $\Gamma$-graded subalgebra
\begin{equation}\label{intwin}
\text{End}^{\text{r}(\Gamma)}_{\mathcal{A}}(\mathcal{W})\subseteq \text{End}_{\C}^{\text{r}(\Gamma)}(\mathcal{W}_{})
\end{equation}
of right graded endomorphisms of $\mathcal{W}$, which intertwine with its left $\mathcal{A}$-action. Then $\mathcal{W}$ can be regarded as a $\Gamma$-graded
$(\mathcal{A},\text{End}^{\text{r}(\Gamma)}_{\mathcal{A}}(\mathcal{W}))$-bimodule.

If $\mathcal{W}_{}$ is finite-dimensional, or more generally, if supp$_{\Gamma}(\mathcal{W})$ is finite,
then both graded algebras \eqref{grend} coincide with the ungraded endomorphism algebra $\text{End}_{\C}(\mathcal{W}_{})$ (see \cite[P. 2]{EK13}).

With the above notation, here is a description of induced gradings and their support in more details.
Let $V$ be an $r$-dimensional vector space. As usual, this space is regarded as an End$_{\C}(V)$-module.
Let \eqref{eq:vsgrading} be a $\Gamma$-grading of $V$. Can the group $\Gamma$ grade the algebra End$_{\C}(V)$ in such a way that $V$ is a left graded module over it?
Indeed, the $\Gamma$-graded space \eqref{eq:vsgrading} well-determines a character \eqref{char}.
Let \begin{equation}\label{tuple}
(\gamma_1,\cdots,\gamma_r)\in \Gamma^r
\end{equation}
be any $r$-tuple of group elements, such that $\sum _{i=1}^r \gamma _i$ is equal to the character $x\in \mathbb{N}[\Gamma]$
(evidently, the tuple can be ordered in few ways). Then there exists a base $\mathcal{E}:=\{e_1,\cdots,e_r\}$ of $V$ such that $e_i\in V_{\gamma_i}$ for every $i=1,\cdots,r$.

Let $E_{i,j}$ be the elementary matrix whose $(i,j)$-th entry is 1, whereas all its other entries are 0. Considering this matrix as a $V$-endomorphism under
the choice of the base $\mathcal{E}$, then
\begin{eqnarray}\label{Eij}
E_{i,j}:\left\{
\begin{array}{rcl} V_{\gamma_j}&\to& V_{\gamma_i}, \text{ and }\\
 V_{\gamma_l}&\to& 0, \text{   if   }l\neq j.
\end{array}\right.
\end{eqnarray}

Impose now the demand that $V$ should be a $\Gamma$-graded left module over End$_{\C}(V)$. Then from \eqref{Eij} it follows that $E_{i,j}$
must be homogeneous of degree $\gamma_i\cdot\gamma _j^{-1}$.
Then under this $\Gamma$-grading the elementary matrices form a $\C$-basis of $ M_r(\C)$ consisting of homogeneous elements. Thus, the elementary $\Gamma$-grading
 $$\mathcal{G}_{ M_r(\C)}:
 \begin{array}{ccl}
 M_r(\C)&=&\bigoplus_{\gamma\in \Gamma}\left(M_r(\C)\right)_{\gamma},\ \ \text{ where }\\
 \left(M_r(\C)\right)_{\gamma}&=&\text{span}_{\C}\left\{E_{i,j}|\ \  \gamma=\gamma_i\cdot\gamma _j^{-1}\right\}
 \end{array}$$
admits $V$ as a graded left module, as well as $x$ as its character.

More generally, let $$\mathcal{A}^r=\bigoplus_{i=1}^r\mathcal{A}_i=\bigoplus_{i=1}^re_i\cdot\mathcal{A}$$
be a free (right) $\mathcal{A}$-module of rank $r$ with basis $e_1,\cdots,e_r$.
As customary, the rule
\begin{equation}\label{Etens}
E_{i,j}\otimes_{\C} a:\ \ \ e_l\cdot a'\mapsto\left\{
\begin{array}{rl} e_i\cdot a\cdot a',& \text{ if } l=j \\
 0,& \text{   if   }l\neq j
\end{array}\right.,\ \ a,a'\in\mathcal{A}
\end{equation}
furnishes $\mathcal{A}^r$ with a left $M_r(\C)\otimes_{\C}\mathcal{A}$-module structure. Clearly, right multiplication in the algebra $\mathcal{A}$ intertwines with the left action \eqref{Etens}.
Suppose now that $\mathcal{A}$ possesses a \eqref{eq:algebragrading}, and let \eqref{tuple} be an $r$-tuple. Define the following two $\Gamma$-gradings.
\begin{enumerate}
\item Grade $\mathcal{A}^r$ by letting $e_j\cdot\mathcal{A}_{\gamma}$ be homogeneous of degree $\gamma_j\cdot\gamma$.
\item Grade $M_{r}(\mathcal{A}) \cong M_r(\C)\otimes_{\C} \mathcal{A}$ by letting $E_{i,j}\otimes \mathcal{A}_{\gamma}$ be homogeneous of degree $\gamma_i\cdot\gamma\cdot\gamma _j^{-1}$.
\end{enumerate}
Then these two gradings are compatible, that is $\mathcal{A}^r$ is a $\Gamma$-graded left module over $M_r(\mathcal{A})\cong M_r(\C)\otimes_{\C} \mathcal{A}$ under the action \eqref{Etens}.
The $\Gamma$-grading (2) of $M_{r}(\mathcal{A})$ is \textit{induced} from the grading \eqref{eq:algebragrading} of $\mathcal{A}$ by $\sum _{i=1}^r\gamma_i\in\N[\Gamma]$, and is denoted by
$\sum _{i=1}^r\gamma_i(\mathcal{G_A}).$

It can then be verified that for every $\gamma_0\in\Gamma$
 \begin{equation}\label{eq:hominduce}
 \dim_{\C}\left(\sum _{\gamma\in\Gamma}n_{\gamma}\gamma(\mathcal{G_A})\right)_{\gamma_0}=
\sum_{\gamma_1\cdot\gamma_2\cdot\gamma_3^{-1}= {\gamma_0}} n_{\gamma_1}\cdot\dim_{\C}(\mathcal{A}_{\gamma_2})\cdot n_{\gamma_3} . \end{equation}
In particular, the support (see \eqref{supp}) of the
induced $\Gamma$-grading class $\left[\sum _{\gamma\in\Gamma}n_{\gamma}\gamma(\mathcal{G_A})\right]$ is given by
(see \cite[\S 4]{MR1941224})
\begin{equation}\label{eq:suppinduce}
\text{supp}_{\Gamma}
\left(\sum _{\gamma\in\Gamma}n_{\gamma}\gamma(\mathcal{G_A})\right)=\left\{\gamma_1\cdot\gamma_2\cdot\gamma_3^{-1}|\ \gamma_2\in\text{supp}_{\Gamma}(\mathcal{G_A}),\ \ n_{\gamma_1},n_{\gamma_3}>0\right\}.
\end{equation}
Since $x(\mathcal{G_A})$ grades the algebra $M_r(\mathcal{A})$, where $x\in \N[\Gamma]$ is of augmentation $\epsilon(x)=r$, then
\begin{equation}\label{augdim}
\dim_{\C}\left(x(\mathcal{G_A})\right)=\epsilon(x)^2\cdot\dim_\C(\mathcal{A}).
\end{equation}
Graded induction respects the product in the semiring $\N[\Gamma]$ as follows
$$\left[x_1(x_2(\mathcal{G}_{\mathcal{A}}))\right]=\left[x_1\cdot x_2(\mathcal{G}_{\mathcal{A}})\right],\ \  \forall x_1,x_2\in \N[\Gamma].$$

For later use we recall that given a subgroup $G < \Gamma$, the {\it left $G$-coset} (see \cite[Definition 2.1]{GS16}) of an element $x=\sum_{i=1}^r \gamma_i\in \N [\Gamma]$
is the set of elements
\begin{equation}\label{coset}
xG:=\left\{\sum_{i=1}^r \gamma_ig_i\right\}_{g_i\in G}.
\end{equation}
Two elements $x,x^{\shortmid}\in \N [\Gamma]$ are left $G$-equivalent
if they belong to the same left $G$-coset of $\N [\Gamma]$, that is
if $xG=x^{\shortmid} G$. The corresponding quotient set is denoted
by $\bigslant{\N [\Gamma]}{R_G}$. We remark that when $G\lhd \Gamma$ is normal,
then the quotient set admits a semiring structure that can be identified with the group semiring
$\N[\Gamma/G]$.

\subsection{Gradings of semisimple algebra and their quotients}\label{gsa}
\begin{definition}\label{def:simplygraded}
A {\it graded ideal} of a group-graded algebra~\eqref{eq:algebragrading} is a graded submodule of the two-sided free module $\mathcal{A}_{}.$
In other words, it is a two-sided ideal $I$ of $\mathcal{A}_{}$ satisfying
$$I=\bigoplus _{\gamma\in \Gamma} I\cap \mathcal{A}_{\gamma}.$$
A graded algebra is {\it graded-simple} if it admits no
non-trivial graded ideals.
\end{definition}
The following is an important observation.
\begin{theorem}\label{AW}(see e.g. \cite[Theorem 2.3']{CM}, \cite[\S A.$1$]{NVO82})
    Any group-grading of a semisimple finite-dimensional algebra admits a decomposition as a direct sum of graded-simple algebras.
\end{theorem}
TGA gradings, as well as gradings which are induced from such gradings, are natural
examples of $G$-simple gradings. The converse is also true due to
\begin{theorem}\label{BSZ}
\cite[Theorem 5.1]{MR1941224},\cite[Theorem 3]{MR2488221}
Any $\Gamma$-simple grading class of a finite-dimensional complex algebra is of the form $[x(\C^{\alpha}G)]$, i.e. it is induced by some $x\in\N[\Gamma]$
from a TGA grading $\C^{\alpha}G$ for some subgroup $G<\Gamma$ and $\alpha \in Z^2(G,\C ^*)$.
\end{theorem}
The \textit{elementary} part $x\in\N[\Gamma]$ and the \textit{fine} part $\C^{\alpha}G$ of the simple grading in Theorem \ref{BSZ} are
determined up to a certain equivalence relation in \cite[Corollary 2.22]{EK13} (see also \cite[Theorem 2.20]{GS16}).

By Theorem \ref{AW} and Theorem \ref{BSZ}, any group grading of a complex
semisimple algebra is obtained from TGA gradings by direct
sums and inductions.
More explicitly, any grading class $[\mathcal{G_A}]$ of a finite-dimensional semisimple complex algebra $\mathcal{A}$ by a group $\Gamma$ is a direct sum of, say, $l$ simply-graded classes,
each of which are induced from TGA grading classes.
That is, for every $i=1,\cdots,l$ there exist
\begin{enumerate}
    \item a subgroup $G_i< \Gamma,$ supporting the fine part of the $i$-th summand,
    \item a 2-cocycle $\alpha_i\in Z^2(G_i,\C^*),$ and
    \item an element $x_i\in \N [\Gamma]$, responsible for the elementary part of the $i$-th summand,
\end{enumerate}
decomposing $[\mathcal{G_A}]$ as \eqref{generalsum}.

Let \eqref{generalsum} be a $\Gamma$-grading class of a finite-dimensional semisimple algebra $\mathcal{A}$.
Let $N\lhd \Gamma$ be any normal subgroup.
By Theorem \ref{AW}, the corresponding quotient class $[\bigslant{\mathcal{G_A}}{N}]$ is a direct sum of quotients of its graded simple summands.
Furthermore, since induction commutes with quotient morphisms \cite[Lemma 2.14]{GS16}, then $[\bigslant{\mathcal{G_A}}{N}]$ admits the decomposition \eqref{reduct}.
Here $\overline{x_i}\in\N [\Gamma/N]$ are just $x_iN\in \bigslant{\N[\Gamma]}{R_N}$ (see \eqref{coset}).
%under the identification of $\N [\Gamma/N]$ with  natural extension of mod$(N)$ to $\N [\Gamma]$.
It should be remarked that the $l$ summands in \eqref{reduct} are no longer graded simple, and may themselves be decomposable into simply-graded constituents.

\subsection{Crossed products, the inertia group and the obstruction class}\label{CPSec}
A graded algebra \eqref{eq:algebragrading} is a \textit{crossed product} if it admits an invertible element $u_{\gamma}$ in every homogeneous component $\mathcal{A}_{\gamma}$.
A conventional notation   for a crossed product is
\begin{equation}\label{CP}
\mathcal{A}_{\Gamma}=\mathcal{A}_e*\Gamma=\oplus_{{\gamma}\in \Gamma}\mathcal{A}_eu_{\gamma}=\oplus_{{\gamma}\in \Gamma}u_{\gamma}\mathcal{A}_e.
\end{equation}
The invertible homogeneous element $u_{\gamma}\in\mathcal{A}_{\gamma}$ is not unique.
More precisely, all invertible $\gamma$-homogeneous elements are obtained as a product $x\cdot u_{\gamma}$ of $u_{\gamma}$ with all invertible elements $x\in\mathcal{A}_{e}^*$.

Twisted group algebras \eqref{tga} are examples of crossed products. As can easily be verified, quotient gradings of crossed products are also crossed products
(while non-trivial quotient gradings of TGA gradings are not TGA gradings).

Let \eqref{CP} be a crossed product and let $M$ be a left module over its base algebra $A_e$.
Then the underlying $\C$-space $M$ can be endowed with another $\mathcal{A}_e$-module structure $\star^{\gamma}$ for any $\gamma\in \Gamma$ as follows
\begin{equation}\label{star}
a\star^{\gamma} m:=u_{\gamma}^{-1}au_{\gamma}(m),\ \  \ \ a\in A_e, m\in M.
\end{equation}
Denoting this module by $^{\gamma}M$, the next two claims are easily checked.
\begin{lemma}\label{eqact}
    Let \eqref{CP} be a crossed product and let $M$ be a left $A_e$-module. Then with the notation \eqref{star},
\begin{enumerate}
\item a different choice of the invertible homogeneous element $u_{\gamma}$ in \eqref{star} yields another $\mathcal{A}_e$-module isomorphic to $^{\gamma}M$.
  \item  there is an $\mathcal{A}_e$-module isomorphism
    \begin{eqnarray}\label{actM}
    \begin{array}{ccc}
    ^{\gamma}M&\to &\mathcal{A}_{\gamma}{\otimes}_{\mathcal{A}_e}M\\
    m&\mapsto &u_{\gamma}{\otimes}_{\mathcal{A}_e}m
    \end{array},
    \end{eqnarray}
where the r.h.s. of \eqref{actM} possesses the natural left $\mathcal{A}_e$-module structure by multiplication.
    \end{enumerate}
\end{lemma}
Lemma \ref{eqact} enables us to identify the module structures of both sides of \eqref{actM}.
It is not hard to check that the crossed product structure \eqref{CP} gives rise to a well-defined left action
    \begin{eqnarray}\label{act}
    \begin{array}{ccc}
    \Gamma\times\text{Mod}(\mathcal{A}_e)&\to&\text{Mod}(\mathcal{A}_e)\\
    (\gamma,[M])&\mapsto & \gamma([M]):=[^{\gamma}M]=[\mathcal{A}_{\gamma}{\otimes}_{\mathcal{A}_e}M]
    \end{array}
    \end{eqnarray}
of the grading group $\Gamma$ on the isomorphism types Mod$(\mathcal{A}_e)$ of $\mathcal{A}_e$-modules.
In turn, the subset Irr$(\mathcal{A}_e)\subseteq\text{Mod}(\mathcal{A}_e)$ of isomorphism types of simple $\mathcal{A}_e$-modules is stable under \eqref{act},
determining an action  of $\Gamma$ on Irr$(\mathcal{A}_e)$.
The fact that $M$ and $^{\gamma}M$ share the same underlying $\C$-space ascertains that
\begin{lemma}\label{samedim}
A $\Gamma$-orbit of $\text{Mod}(\mathcal{A}_e)$ under the action \eqref{act}
well-determines a dimension of its members,
that is dim$_{\C}(^{\gamma}M)=$dim$_{\C}(M)$ for every ${\gamma}\in\Gamma$.
\end{lemma}
It should be remarked that unlike our crossed product case, the dimensions of $M$ and $\mathcal{A}_{\gamma}{\otimes}_{\mathcal{A}_e}M$ may be unequal even when the algebra
\eqref{eq:algebragrading} is \textit{strongly graded}.
    \begin{definition}\label{basealginertdef}\cite[\S 3.1]{YG21}
        Let \eqref{CP} be a crossed product and let $M$ be a left $\mathcal{A}_e$-module.
        The {\it inertia group} of $M$ with respect to the grading \eqref{CP} is the stabilizer of $[M]$ under the action \eqref{act}, that is
        \begin{equation}\label{inertiasuspension}
\mathcal{I}_{\mathcal{G}}(M):=
\{\gamma\in \Gamma |\ [M]=[\gamma(M)]\}=
\{\gamma\in \Gamma |\  M \cong ^{\gamma}M{\cong}{\mathcal{A}_{\gamma}}{\otimes}_{\mathcal{A}_e}M\}.
        \end{equation}
        In particular, $[M]$ is {\it invariant} with respect to \eqref{CP} if $\mathcal{I}_{\mathcal{G}}(M)=\Gamma$.
    \end{definition}
Let \eqref{CP} be a complex crossed product, and let $M$ be a left $\mathcal{A}_e$-module.
Then ${\mathcal{A}_{e}*\Gamma}{\otimes}_{\mathcal{A}_e}M$ admits a left $\Gamma$-graded $\mathcal{A}_{e}*\Gamma$-module structure.
The main player in the next theorem is the $\Gamma$-graded subalgebra
$$\text{End}^{\text{r}(\Gamma)}_{\mathcal{A}_{e}*\Gamma}({\mathcal{A}_{e}*\Gamma}{\otimes}_{\mathcal{A}_e}M)\subseteq
\text{End}^{\text{r}(\Gamma)}_{\C}({\mathcal{A}_{e}*\Gamma}{\otimes}_{\mathcal{A}_e}M)$$ of right graded endomorphisms of $\mathcal{A}_{e}*\Gamma{\otimes}_{\mathcal{A}_e}M$
(see \eqref{grend} and \eqref{intwin}), which intertwine with its left $\mathcal{A}_{e}*\Gamma$-action.
As noted in Lemma \ref{eqact}(2), the conjugation action, and thus also the corresponding stabilizers,
coincides with the tensor action by homogeneous components for more general gradings, originally used for the following results.
    \begin{theorem}\label{let}(see \cite[Theorem 5.23]{YG21})
Let \eqref{CP} be a complex crossed product, and let $[M]\in$Irr$(\mathcal{A}_e)$. Suppose further that the dimension of $M$ (over $\C$) is enumerable, or more generally is \textit{absolutely irreducible}
(see \cite[\S 2.2]{YG21}).
        Then there exists a unique cohomology class $\omega_{\mathcal{G}}([M])\in H^2(\mathcal{I}_{\mathcal{G}}(M),\C^*)$ of the inertia $\mathcal{I}_{\mathcal{G}}(M)<\Gamma$ such that
    $$\text{End}^{\text{r}(\Gamma)}_{\mathcal{A}_{e}*\Gamma}({\mathcal{A}_{e}*\Gamma}{\otimes}_{\mathcal{A}_e}M)\stackrel{\Gamma}{\cong} \C^{\omega}\mathcal{I}_{\mathcal{G}}(M)$$
    for any 2-cocycle $\omega$ in the class $\omega_{\mathcal{G}}([M])$.
    \end{theorem}
The cohomology class $\omega_{\mathcal{G}}([M])\in H^2(\mathcal{I}_{\mathcal{G}}(M),\C^*)$ is termed {\it Mackey's obstruction cohomology class} (which has a much broader notion, see \cite{YG21}).
Owing to Theorem \ref{let}, the twisted group algebra
$\C^{\omega_{}}\mathcal{I}_{\mathcal{G}}(M)$ can be regarded as acting on ${\mathcal{A}_{e}*\Gamma}{\otimes}_{\mathcal{A}_e}M$ from the right (intertwining with the $\mathcal{A}_{e}*\Gamma$-action).
Since $\C^{\omega_{}}\mathcal{I}_{\mathcal{G}}(M)$ is a graded division algebra, that is all its non-zero homogeneous elements are invertible,
its graded module ${\mathcal{A}_{e}*\Gamma}{\otimes}_{\mathcal{A}_e}M$ is free and thus admits a well-defined dimension (see \cite[\S 2.1]{EK13}).
As above, consider the $\Gamma$-graded subalgebra
$$\text{End}^{\text{l}(\Gamma)}_{\C^{\omega_{}}\mathcal{I}_{\mathcal{G}}(M)}({\mathcal{A}_{e}*\Gamma}{\otimes}_{\mathcal{A}_e}M)\subset
\text{End}^{\text{l}(\Gamma)}_{\C}({\mathcal{A}_{e}*\Gamma}{\otimes}_{\mathcal{A}_e}M)$$
of left graded endomorphisms of $\mathcal{A}_{e}*\Gamma{\otimes}_{\mathcal{A}_e}M$ (see \eqref{grend}), which intertwine with its right $\C^{\omega_{}}\mathcal{I}_{\mathcal{G}}(M)$-action.
       \begin{lemma}\label{isogr}(see \cite[Theorem I.5.8]{NVO82} and \cite[Corollary D]{YG21})
With the above notation, the algebra
$\text{End}^{\text{l}(\Gamma)}_{\C^{\omega_{}}\mathcal{I}_{\mathcal{G}}(M)}({\mathcal{A}_{e}*\Gamma}{\otimes}_{\mathcal{A}_e}M)$ of endomorphisms is $\Gamma$-graded simple, such that
        \begin{equation}\label{Endgris}
\text{End}^{\text{l}(\Gamma)}_{\C^{\omega_{}}\mathcal{I}_{\mathcal{G}}(M)}(\mathcal{A}_{e}*\Gamma{\otimes}_{\mathcal{A}_e}M)\stackrel{{\Gamma}}{\cong}y_M(\C^{\omega_{}}\mathcal{I}_{\mathcal{G}}(M))
        \end{equation}
for some $y_M\in \N [\Gamma]$ (which is not uniquely determined, see \cite[Theorem 2.20]{GS16}) of augmentation
$\epsilon(y_M)=\dim_{\C^{\omega_{}}\mathcal{I}_{\mathcal{G}}(M)}({\mathcal{A}_{e}*\Gamma}{\otimes}_{\mathcal{A}_e}M)$.
    \end{lemma}
Remaining with the above notation, for any $a\in \mathcal{A}_{e}*\Gamma$, let
    \begin{eqnarray}\label{varphix}
    \varphi_{a}:
    \begin{array}{rcr}{\mathcal{A}_{e}*\Gamma}{\otimes}_{\mathcal{A}_e}M&\to& {\mathcal{A}_{e}*\Gamma}{\otimes}_{\mathcal{A}_e}M\\
    u_{\gamma}{\otimes}_{\mathcal{A}_e}m&\mapsto&a\cdot u_{\gamma}{\otimes}_{\mathcal{A}_e}m
    \end{array}
    \end{eqnarray}
denote an action of $a$ on the left $\mathcal{A}_{e}*\Gamma$-module ${\mathcal{A}_{e}*\Gamma}{\otimes}_{\mathcal{A}_e}M$.
Then there is a $\Gamma$-graded homomorphism (see \cite[Lemma 3.23]{YG21})
    \begin{eqnarray}\label{varphixg}\varphi_M:
    \begin{array}{ccc}
    \mathcal{A}_{\Gamma}&\stackrel{\Gamma}{\to}&\text{End}^{\text{l}(\Gamma)}_{\C^{\omega_{}}\mathcal{I}_{\mathcal{G}}(M)}(\mathcal{A}_{e}*\Gamma{\otimes}_{\mathcal{A}_e}M)\\
    a&\mapsto&\varphi_{a}
    \end{array}.
    \end{eqnarray}
The following is an immediate consequence of Theorem \ref{let} and the Graded Density Theorem (see \cite[Theorem 2.5]{EK13}).
    \begin{theorem}\label{C}\cite[Theorem C]{YG21}
        Let \eqref{CP} be a finite-dimensional complex crossed product, and let $M$ be any simple left $\mathcal{A}_e$-module.
        Then the simply $\Gamma$-graded endomorphism algebra
        $$\text{End}^{\text{l}(\Gamma)}_{\C^{\omega_{}}\mathcal{I}_{\mathcal{G}}(M)}(\mathcal{A}_{e}*\Gamma{\otimes}_{\mathcal{A}_e}M)$$
        is a graded image of $\mathcal{A}_e*\Gamma$ via \eqref{varphixg} for any 2-cocycle $\omega_{}$ in the obstruction cohomology class
        $\omega_{\mathcal{G}}([M])\in H^2(\mathcal{I}_{\mathcal{G}}(M),\C^*)$ of the inertia $\mathcal{I}_{\mathcal{G}}(M)<\Gamma$.
    \end{theorem}

\section{On equi-dimensional gradings}\label{equi}
\begin{definition}\label{equidef}
A $\Gamma$-graded algebra \eqref{eq:algebragrading} is
{\it equi-dimensional} if\\ dim$_{\C}(\mathcal{A}_{\gamma_1})=$dim$_{\C}(\mathcal{A}_{\gamma_2})$ for every
$\gamma_1,\gamma_2\in {\Gamma}$. A  $\Gamma$-grading class is equi-dimensional if so are its members.
\end{definition}
Crossed product grading classes, including TGA grading classes, are equi-dimensional. Furthermore, it is
easy to verify that the family of equi-dimensional grading classes is closed to quotients. Therefore, this property is helpful in the proof of Theorem A.
Here is a sufficient and necessary condition for a graded class induced from an equi-dimensional graded class to be equi-dimensional itself.
\begin{theorem}\label{th:equidim}
Let $\Gamma$ be a group, let $G<\Gamma$ be a subgroup of finite index and let $[\mathcal{G_A}]$ be an equi-dimensional $G$-graded class of an algebra $\mathcal{A}$.
Then the induced $\Gamma$-grading class
$\left[\sum_{\gamma\in \Gamma}n_{\gamma}\gamma(\mathcal{G_A})\right]$ is equi-dimensional if and only if
$$m_{\gamma_1 G}=m_{\gamma_2 G},\ \  \forall \gamma_1,\gamma_2\in \Gamma,$$
where $m_{\gamma G}:=\sum_{\gamma' \in \gamma G}n_{\gamma'}$ for the left coset $\gamma G\in[\Gamma:G]$.
\end{theorem}
\begin{proof}
Let $\mathcal{B}$ be the corresponding matrix algebra over $\mathcal{A}$ endowed with the induced grading $\sum_{\gamma\in \Gamma}n_{\gamma}\gamma(\mathcal{G_A})$.
We compute dim$_{\C}(\mathcal{B}_{\gamma_0})$ for $\gamma _0\in \Gamma$.
Let $d:=$dim$_{\C}(\mathcal{A}_g)$ be the common dimension of all components of $\mathcal{A}$.
Then by~\eqref{eq:hominduce}, we have $$\dim_{\C}(\mathcal{B}_{\gamma_0})=\sum n_{\gamma _1}\cdot d \cdot n_{\gamma_2},$$
where the sum is over all expressions of the form
\begin{equation}\label{expr}\gamma _0=\gamma_1\cdot g \cdot\gamma _2^{-1},\ \ g\in G,\gamma_1 ,\gamma _2\in \Gamma.
\end{equation}
With the notation \eqref{sigamma} notice that expressions of the form \eqref{expr} satisfy
$\gamma_1\cdot g\in\sigma_{\gamma_0}(\gamma_2G).$
Hence, the contribution to dim$_{\C}(\mathcal{B}_{\gamma_0})$ for a fixed $\gamma _2$ is
    \begin{equation}\label{eq:gamma1contr}
 \sum_{\gamma\in \sigma_{\gamma _0}(\gamma_2G)}n_{\gamma}\cdot d\cdot n_{\gamma _2}.
    \end{equation}
Notice that the number \eqref{eq:gamma1contr} does not depend on the representative $\gamma_2$ in its left coset. Therefore, the contribution of this entire coset $\gamma_2G$
to dim$_{\C}(\mathcal{B}_{\gamma_0})$ is
            \begin{equation}\label{eq:gamma1cosetcontr}
 \sum_{\gamma\in \sigma_{\gamma _0}(\gamma_2G)}n_{\gamma}  \cdot d\cdot  \sum_{\gamma'\in\gamma _2G} n_{\gamma'}=m_{\sigma_{\gamma _0}(\gamma_2 G)}\cdot d\cdot m_{\gamma_2 G}.
        \end{equation}
Summing~\eqref{eq:gamma1cosetcontr} over all the cosets in $[\Gamma:G]$ we have
        \begin{equation}\label{eq:dimofgamma0}
\text{dim}_{\C}(\mathcal{B}_{\gamma_0})=d\cdot\sum _{\gamma G\in [\Gamma :G]}m_{\gamma G}\cdot m_{\sigma_{\gamma _0}(\gamma G)}.
        \end{equation}
From \eqref{eq:dimofgamma0} it is clear that if $m_{\gamma_1 G}=m_{\gamma_2 G}(=m$, say) for all $\gamma_1,\gamma_2\in \Gamma$, then
$$\text{dim}_{\C}(\mathcal{B}_{\gamma_0})=d\cdot m^2\cdot|[\Gamma :G]|,\ \ \forall \gamma_0\in \Gamma.$$
In particular, the grading on $\mathcal{B}$ is equi-dimensional, proving the ``if" direction.

Note that $\sigma_e=$Id$_{[\Gamma:G]}$ for the identity element $e\in \Gamma$. Hence by~\eqref{eq:dimofgamma0} we have
        \begin{equation}\label{eq:dimofe}
\text{dim}_{\C}(\mathcal{B}_{e})=d\cdot\sum _{\gamma G\in [\Gamma :G]}m_{\gamma G}^2=
d\cdot\sqrt{\sum _{\gamma G\in [\Gamma :G]}m^2_{\gamma G}}\cdot \sqrt{\sum _{\gamma G\in [\Gamma :G]}m^2_{\gamma G}}.
    \end{equation}
By reordering the summation under the square root on the right hand side of~\eqref{eq:dimofe} with respect to the permutation $\sigma_{\gamma _0}$ for any $\gamma_0\in \Gamma$ we have
    \begin{equation}\label{eq:dimofe2}
\text{dim}_{\C}(\mathcal{B}_{e})=d\cdot
\sqrt{\sum _{\gamma G\in [\Gamma :G]}m^2_{\gamma G}}\cdot \sqrt{\sum _{\gamma G\in [\Gamma :G]}m^2_{\sigma_{\gamma _0}(\gamma G)}},\ \ \ \forall \gamma_0\in \Gamma.
    \end{equation}
In particular, by the Cauchy-Schwarz inequality, the dimension \eqref{eq:dimofgamma0} does not exceed the dimension \eqref{eq:dimofe2}, i.e. (compare with \cite[Theorem B]{ginosar2013graph})
\begin{equation}\label{egeq}
\text{dim}_{\C}(\mathcal{B}_{\gamma_0})\leq\text{dim}_{\C}(\mathcal{B}_{e}), \ \ \ \forall \gamma_0\in \Gamma.
\end{equation}

Now, assume that the grading on $\mathcal{B}$ is equi-dimensional. In particular, the dimensions in~\eqref{eq:dimofgamma0} and~\eqref{eq:dimofe2} are equal for every $\gamma_0\in \Gamma$, and so
$$d\cdot\sum _{\gamma G\in [\Gamma :G]}m_{\gamma G}\cdot m_{\sigma_{\gamma _0}(\gamma G)}
=d\cdot\sqrt{\sum _{\gamma G\in [\Gamma :G]}m^2_{\gamma G}}\cdot \sqrt{\sum _{\gamma G\in [\Gamma :G]}m^2_{\sigma_{\gamma _0}(\gamma G)}},\ \ \ \forall \gamma_0\in \Gamma.$$
Consequently, the equality condition in the Cauchy-Schwarz inequality implies that for every $\gamma_0\in\Gamma$, the $|[\Gamma :G]|$-tuples
$$\left(m_{\gamma G}\right)_{\gamma G\in [\Gamma :G]}  \text{   and   } (m_{\sigma _{\gamma _0}(\gamma G)})_{\gamma G\in [\Gamma :G]}$$
are $\mathbb{Q}$-dependent. Since these are tuples of non-negative integers that defer by a permutation, we conclude from their linear dependence that they are equal. Finally, since the
$\Gamma$-action \eqref{psigamma} on\\ $[\Gamma:G]$ is transitive, these tuples are constant, that is $m_{\gamma _1G}=m_{\gamma_2G}$ for any\\
$\gamma_1,\gamma_2\in \Gamma$. This proves the``only if" part.
\end{proof}
With the notation of Theorem \ref{th:equidim} choose a transversal set $\mathcal{T}$ of $G$ in $\Gamma$. Then the theorem says that the induced $\Gamma$-grading is equi-dimensional if and only if
the element responsible for the induction can be given independently of $\gamma_0 \in \Gamma$ as
\begin{equation}\label{respind}
\sum_{\gamma \in \Gamma }n_{\gamma}=m_{\gamma_0 G}\cdot \sum_{t\in \mathcal{T}}t \in\N[\Gamma].
\end{equation}

As mentioned in the beginning of this section, the equi-dimension property is closed to quotients. On the other hand,
by Theorem \ref{AW}, a semisimple finite-dimensional graded algebra decomposes into a direct sum of graded-simple algebras.
The following claim merges these two facts.
\begin{lemma}\label{equidimquot}
Let \eqref{eq:algebragrading} be an equi-dimensional $\Gamma$-grading of a semisimple finite-dimensional algebra $\mathcal{A}$ and let $N\lhd\Gamma$.
Then every simply-graded summand of the quotient $\bigslant{\mathcal{G}_\mathcal{A}}{N}$ is equi-dimensional over $\Gamma/N$.
\end{lemma}
\begin{proof}
Let $\Lambda_{\Gamma/N}$ be a simply-graded summand of the quotient $\bigslant{\mathcal{G}_\mathcal{A}}{N}$. We make use either of \cite[Theorem B]{ginosar2013graph}, or of \eqref{egeq}
(noticing that by Theorem \ref{BSZ}, $\Lambda_{\Gamma/N}$ is induced from some TGA grading which is equi-dimensional) to deduce that
\begin{equation}\label{egeqgraph}
\text{dim}_{\C}(\Lambda_{g})\leq\text{dim}_{\C}(\Lambda_{e}),\ \ \forall g\in\Gamma/N.
\end{equation}
Finally, the equi-dimensional property of the entire sum yields an equality in \eqref{egeqgraph} for every $g\in\Gamma/N$ running over all direct summands of $\bigslant{\mathcal{G}_\mathcal{A}}{N}$.
\end{proof}

\section{Quotient classes of twisted group algebra grading classes}\label{proofA+}
Returning to the setup of Theorem A, let $\C^{\alpha}G$ be a twisted group algebra over a finite group $G$, let $N\lhd G$ be a normal subgroup, and keep the notation $\alpha$ also
for the restriction of this 2-cocycle to $N$.
Our concern is the quotient $G/N$-grading $\mathcal{G}$ of the TGA grading $\C^{\alpha}G$. The base algebra of this $G/N$-graded algebra is just the twisted group algebra $\C^{\alpha}N$.
In the spirit of \eqref{CP}, this crossed product class can be written as
\begin{equation}\label{TGAquot}
[\mathcal{G}]=\left[\bigslant{\C^{\alpha}G}{N}\right]= \left[(\C^{\alpha}N)*(G/N)\right].
\end{equation}
It is convenient to work with the central idempotents of the semisimple base algebra $\C^{\alpha}N$ that correspond to the isomorphism types of irreducible $\C^{\alpha}N$-modules.
%Evidently, there is a one-to-one correspondence between the isomorphism types of irreducible $\C^{\alpha}N$-modules and the primitive central
%idempotents of $\C^{\alpha}N$. More precisely, f
For any $[M]\in$Irr$(\C^{\alpha}N)$ there exists a unique primitive central idempotent $\iota_M\in\C^{\alpha}N$ such that for every
$[M],[M']\in$Irr$(\C^{\alpha}N)$
\begin{eqnarray}\label{M'M}
\iota_{M'}(M)=\left\{
\begin{array}{cc}
M&\text{if  } [M]=[M']\\
0&\text{if  } [M]\neq[M'].
\end{array}\right.
\end{eqnarray}
Let $g\in G$. How does the primitive central idempotent $\iota_{M'}\in\C^{\alpha}N$ operate on the module $^gM$ (see \eqref{star})?
Clearly, $a:=u_{g}\iota_{M'}u_{g}^{-1}$ is a primitive central idempotent of $\C^{\alpha}N$. Plugging this $a$ in \eqref{star} for $\gamma:=g$ we have
\begin{equation}\label{staridemp}
u_{g}\iota_{M'}u_{g}^{-1}\star^g m=\iota_{M'}(m),
\end{equation}
where the r.h.s. is the action in $M$ and the l.h.s. is the action in $^gM$ ($m$ belongs to both modules which have the same underlying vector space).

By equations \eqref{M'M} and \eqref{staridemp} we obtain for every $g\in G$ and $[M]\in$Irr$(\C^{\alpha}N)$
\begin{equation}\label{iotagM}
\iota_{(^gM)}=u_g\cdot\iota_M\cdot u_g^{-1}\in\C^{\alpha}N.
\end{equation}
We deduce
\begin{lemma}\label{iotalemma}
    With the above notation $[M],[M']\in$Irr$(\C^{\alpha}N)$ belong to the same orbit under the action \eqref{act} if and only if
    \begin{equation}\label{iotacond}
    \iota_{M'}\cdot\C^{\alpha}G\cdot\iota_{M}\neq 0.
    \end{equation}
\end{lemma}
\begin{proof}
    Considering the graded basis \eqref{eq:natgradtga} of $\C^{\alpha}G$, equation \eqref{iotacond} holds if and only if there exists $g\in G$ such that $\iota_{M'}\cdot u_g\cdot\iota_{M}\neq 0.$
    Applying \eqref{iotagM}, the latter condition is saying exactly that there exists $g\in G$ such that $\iota_{M'}\cdot \iota_{(^gM)}\neq 0.$
    In turn, by orthogonality of the primitive central idempotents of $\C^{\alpha}N$ this is the same as the existence of $g\in G$ such that $[M']=[ ^gM],$
    that is, $[M]$ and $[M']$ lie in the same $G$-orbit.
\end{proof}

\subsection{Proof of Theorem A}\label{proof_of_A}
Let \eqref{tga} be a TGA grading,
where $\alpha\in Z^2(G,\C^*)$ is a 2-cocycle, and let $N\lhd G$ a normal subgroup.
By \eqref{TGAquot}, the quotient grading $\bigslant{\C^{\alpha}G}{N}$ of the TGA grading
$\C^{\alpha}G$ is $G/N$-graded isomorphic to a crossed product of the group ${G/N}$ over
the base algebra $(\bigslant{\C^{\alpha}G}{N})_e=\C^{\alpha}N$. For any full set $\mathcal{T}_{G/N}$ of representatives of $G/N$ in $G$ we have
\begin{equation}\label{CfG}
\bigslant{\C^{\alpha}G}{N}\stackrel{G/N}{\cong}(\C^{\alpha}N)*{G/N}=\oplus_{g\in \mathcal{T}_{G/N}}(\C^{\alpha}N)u_{g}.
\end{equation}
As from this point, the group $G$ is assumed to be finite.
For any $[M]\in \text{Irr}(\C ^{\alpha}N)$ let $\mathcal{I}_{\mathcal{G}}({M})< {G/N}$ be the stabilizer of $[M]$ under the action \eqref{act},
and let $T_{M}$ be a transversal set of $\mathcal{I}_{\mathcal{G}}({M})$ in ${G/N}$.
Then the orbit of $[M]$ in Irr$(\C ^{\alpha}N)$ under the ${G/N}$-action is given by $\left\{[^tM]\right\}_{t\in T_{M}}$.
Let $\Omega\subset\text{Irr}(\C ^{\alpha}N)$ be a set of representatives for the ${G/N}$-orbits of isomorphism types of irreducible $\C^{\alpha}N$-modules.
Then the twisted group algebra $\C^{\alpha}N$ decomposes as
\begin{equation}\label{decFfN}
\C^{\alpha}N\cong\bigoplus_{[M]\in \Omega}\left(\bigoplus_{t\in T_{M}}\text{End}_{\C}(^tM)\right).
\end{equation}

By Theorem \ref{let}, any
irreducible isomorphism type of $\C^{\alpha}N$-module
$[M]\in\text{Irr}(\C ^{\alpha}N)$ gives rise to Mackey's obstruction
cohomology class
$$\omega_{\mathcal{G}}([M])\in H^2(\mathcal{I}_{\mathcal{G}}({M}),\C^*),$$ which is responsible for the fine part of the corresponding simply-graded summand.
With the above notation, we prove Theorem A by showing
    \begin{equation}\label{transgrquotient}
   \bigslant{\C^{\alpha}G}{N}\stackrel{G/N}{\cong}\bigoplus_{[M]\in \Omega}\left(\dim_{\C}(M)\cdot \sum_{t\in T_{M}}t\right)\left(\C^{\omega_{[M]}}\mathcal{I}_{\mathcal{G}}({M})\right),
    \end{equation}
where for every $[M]\in \Omega$, the 2-cocycle $\omega_{[M]}\in Z^2({G/N},\C^*)$ is any member of the class $\omega_{\mathcal{G}}([M])\in H^2({G/N},\C^*)$.

As already mentioned, the twisted group algebra $\C^{\alpha}G$ is semisimple. Hence, all its homomorphic images embed in it as direct summands.
These include, in particular, the ${G/N}$-graded images
$$\text{End}^{\text{l}{(G/N)}}_{\C^{\omega_{[M]}}\mathcal{I}_{\mathcal{G}}(M)}({\C^{\alpha}G}{\otimes}_{\C^{\alpha}N}M),\ \ \forall[M]\in\Irr(\C^{\alpha}N)$$
given in Theorem \ref{C} (putting \eqref{CfG} as $A_e*\Gamma$). Since $M$ is finite-dimensional, then so is $\C^{\alpha}G{\otimes}_{\C^{\alpha}N}M$. Therefore,
the algebra of its graded endomorphisms is the same as the algebra of its ungraded endomorphisms (see \S\ref{inductionsec}). That is, for irreducible $\C^{\alpha}N$-module $M$ we have
\begin{equation}\label{grungr}
\text{End}^{\text{l}({G/N})}_{\C^{\omega_{[M]}}\mathcal{I}_{\mathcal{G}}(M)}({\C^{\alpha}G}{\otimes}_{\C^{\alpha}N}M)=
\text{End}_{\C^{\omega_{[M]}}\mathcal{I}_{\mathcal{G}}(M)}({\C^{\alpha}G}{\otimes}_{\C^{\alpha}N}M).
\end{equation}
The algebra of ungraded endomorphisms on the r.h.s. of \eqref{grungr} is thus regarded as a homomorphic image of $\C^{\alpha}G$ for every $[M]\in\Irr(\C^{\alpha}N)$.
We show that
\begin{enumerate}
    \item If $[M],[M']\in\text{Irr}(\C ^{\alpha}N)$ belong to distinct orbits under \eqref{act}, then
    the direct summands of $\C^{\alpha}G$ corresponding to
    $\text{End}_{\C^{\omega_{[M]}}\mathcal{I}_{\mathcal{G}}(M)}({\C^{\alpha}G}{\otimes}_{\C^{\alpha}N}M)$ and
    $\text{End}_{\C^{\omega_{[M']}}\mathcal{I}_{\mathcal{G}}(M')}({\C^{\alpha}G}{\otimes}_{\C^{\alpha}N}M')$ are distinct.
    \item $\sum_{[M]\in \Omega}\dim_{\C}\left(\text{End}_{\C^{\omega_{[M]}}\mathcal{I}_{\mathcal{G}}(M)}({\C^{\alpha}G}{\otimes}_{\C^{\alpha}N}M)\right)=|G|.$
    \item  For every $[M]\in\text{Irr}(\C ^{\alpha}N)$, the element $y_M$ which is responsible for the induction in \eqref{Endgris} may be written as
    \begin{equation}\label{yM}
    y_M=\dim_{\C}(M)\cdot \sum_{t\in T_{M}}t\in \N[{G/N}].
    \end{equation}
\end{enumerate}
Once (1) is proven, we obtain an embedding
$$\bigoplus_{[M]\in \Omega}\text{End}_{\C^{\omega_{[M]}}\mathcal{I}_{\mathcal{G}}(M)}({\C^{\alpha}G}{\otimes}_{\C^{\alpha}N}M)\hookrightarrow\C^{\alpha}G.$$
Together with (2) it is clear that there is an isomorphism of algebras
\begin{equation}\label{sum}
\C^{\alpha}G\stackrel{}{\cong}\bigoplus_{[M]\in \Omega}\text{End}_{\C^{\omega_{[M]}}\mathcal{I}_{\mathcal{G}}(M)}({\C^{\alpha}G}{\otimes}_{\C^{\alpha}N}M),
\end{equation}
By  \eqref{grungr}, all the summands in the r.h.s. of \eqref{sum} are the same as the left graded endomorphism algebras,
and as such they are ${G/N}$-graded images of the l.h.s., therefore
$$\bigslant{\C^{\alpha}G}{N}\stackrel{G/N}{\cong}
\bigoplus_{[M]\in \Omega}\text{End}^{\text{l}({G/N})}_{\C^{\omega_{[M]}}\mathcal{I}_{\mathcal{G}}(M)}({\C^{\alpha}G}{\otimes}_{\C^{\alpha}N}M)\stackrel{G/N}{\cong}\cdots$$
which by \eqref{Endgris} is of the form
\begin{equation}\label{cdots}
\cdots\stackrel{G/N}{\cong}\bigoplus_{[M]\in \Omega}y_M(\C^{\omega_{[M]}}\mathcal{I}_{\mathcal{G}}(M)).
\end{equation}
Plugging \eqref{yM} in \eqref{cdots} will then complete the proof of \eqref{transgrquotient} and hence also of Theorem A.\\

\textbf{Proof of (1).}
The graded homomorphism \eqref{varphixg} in the setting of Theorem A is
\begin{equation}\label{eq:specific224}
\varphi_{M}:\bigslant{\C^{\alpha}G}{N}\stackrel{G/N}{\to}\text{End}^{\text{l}({G/N})}_{\C^{\omega_{[M]}}\mathcal{I}_{\mathcal{G}}(M)}({\C^{\alpha}G}{\otimes}_{\C^{\alpha}N}M),
\end{equation}
which by Theorem~\ref{C} is surjective.
Compute the image of $\varphi_{M}(\iota_{M'})$, where $\iota_{M'}\in \C^{\alpha}N$
is the central primitive idempotent that corresponds to $M'$.
$$\varphi_{M}(\iota_{M'})({\C^{\alpha}G}{\otimes}_{\C^{\alpha}N}M)=\iota_{M'}\cdot{\C^{\alpha}G}{\otimes}_{\C^{\alpha}N}M=\cdots$$
by \eqref{M'M} $M=\iota_{M}(M)$, hence
$$\cdots=\iota_{M'}\cdot{\C^{\alpha}G}{\otimes}_{\C^{\alpha}N}\iota_{M}(M)=\iota_{M'}\cdot{\C^{\alpha}G}\cdot\iota_{M}{\otimes}_{\C^{\alpha}N}M.$$
By \eqref{iotacond} we conclude that if $[M],[M']\in\text{Irr}(\C ^{\alpha}N)$ belong to distinct orbits under \eqref{act}, then
\begin{equation}\label{eq0}
\varphi_{M}(\iota_{M'})=0\in \text{End}_{\C^{\omega_{[M]}}\mathcal{I}_{\mathcal{G}}(M)}({\C^{\alpha}G}{\otimes}_{\C^{\alpha}N}M).
\end{equation}
Noticing that the surjective map \eqref{eq:specific224} is both a ring homomorphism and a homomorphism of (graded) left $\C^{\alpha}G$-modules, \eqref{eq0} says that
\begin{equation}\label{Im0}
\iota_{M'}\cdot\text{End}_{\C^{\omega_{[M]}}\mathcal{I}_{\mathcal{G}}(M)}({\C^{\alpha}G}{\otimes}_{\C^{\alpha}N}M)=0.
\end{equation}
On the other hand,
$$\varphi_{M'}(\iota_{M'})\in \text{End}_{\C^{\omega_{[M']}}\mathcal{I}_{\mathcal{G}}(M')}({\C^{\alpha}G}{\otimes}_{\C^{\alpha}N}M')$$ is non-zero since for every $m'\in M'$
$$\varphi_{M'}(\iota_{M'})(1{\otimes}_{\C^{\alpha}N}m')=\iota_{M'}{\otimes}_{\C^{\alpha}N}m'=1{\otimes}_{\C^{\alpha}N}\iota_{M'}(m')=1{\otimes}_{\C^{\alpha}N}m'.$$
Again, since \eqref{varphixg} is a homomorphism of left $\C^{\alpha}G$-modules, then
\begin{equation}\label{Imneq0}
\iota_{M'}\cdot\text{End}_{\C^{\omega_{[M']}}\mathcal{I}_{\mathcal{G}}(M')}({\C^{\alpha}G}{\otimes}_{\C^{\alpha}N}M')\neq 0.
\end{equation}
Equations \eqref{Im0} and \eqref{Imneq0} imply that as left $\C^{\alpha}G$-modules
$$\text{End}_{\C^{\omega_{[M]}}\mathcal{I}_{\mathcal{G}}(M)}({\C^{\alpha}G}{\otimes}_{\C^{\alpha}N}M)\ncong
\text{End}_{\C^{\omega_{[M']}}\mathcal{I}_{\mathcal{G}}(M')}({\C^{\alpha}G}{\otimes}_{\C^{\alpha}N}M')$$
for $[M],[M']\in\text{Irr}(\C ^{\alpha}N)$
in distinct orbits under \eqref{act}. Thus, so are their corresponding direct summands.\\

\textbf{Proof of (2).}
Let $[M]\in$Irr$(\C ^{\alpha}N).$
Firstly, the orbit-stabilizer rule says that
\begin{equation}\label{orbstabeq}
|T_M|\cdot|\mathcal{I}_{\mathcal{G}}(M)|=\frac{|G|}{|N|}=|{G/N}|.
\end{equation}
We compute the dimensions corresponding to \eqref{decFfN}.
By Lemma \ref{samedim} the dimensions of all the isomorphism types in an orbit under \eqref{act} are equal. We establish
\begin{equation}\label{dimN}
|N|=\dim_{\C}(\C^{\alpha}N)=\sum_{[M]\in \Omega} \dim_{\C}(M)^2\cdot|T_M|.
\end{equation}
We now handle the complex dimension $\delta_M$ of the simply ${G/N}$-graded summand
$$\text{End}^{\text{l}({G/N})}_{\C^{\omega_{[M]}}\mathcal{I}_{\mathcal{G}}(M)}({\C^{\alpha}G}{\otimes}_{\C^{\alpha}N}M)$$
of $\bigslant{\C^{\alpha}G}{N}$.
Again, by \eqref{grungr} the dimension of this graded endomorphism algebra is the same as the dimension of the corresponding algebra of ungraded endomorphisms.
$$\begin{array}{rl}
\delta_M &:=\dim_{\C}\left(\text{End}_{\C^{\omega_{[M]}}\mathcal{I}_{\mathcal{G}}(M)}({\C^{\alpha}G}{\otimes}_{\C^{\alpha}N}M)\right)\\
&=\dim_{\C}\left(\C^{\omega_{[M]}}\mathcal{I}_{\mathcal{G}}(M)\right)\cdot\left({\text{rank}_{\C^{\omega_{[M]}}\mathcal{I}_{\mathcal{G}}(M)}({\C^{\alpha}G}{\otimes}_{\C^{\alpha}N}M) }\right)^2\\
&=\dim_{\C}\left(\C^{\omega_{[M]}}\mathcal{I}_{\mathcal{G}}(M)\right)\cdot\left(\frac{\dim_{\C}({\C^{\alpha}G}{\otimes}_{\C^{\alpha}N}M)}{\dim_{\C}(\C^{\omega_{[M]}}\mathcal{I}_{\mathcal{G}}(M))}\right)^2\\
&=|\mathcal{I}_{\mathcal{G}}(M)|\cdot\left(\frac{|G|\cdot \dim_{\C}(M)}{|N|\cdot|\mathcal{I}_{\mathcal{G}}(M)|}\right)^2.
\end{array}$$
We conclude
\begin{equation}\label{deps}
\delta_M=\frac{|G|^2\cdot \dim_{\C}(M)^2}{|N|^2\cdot|\mathcal{I}_{\mathcal{G}}(M)|}.
\end{equation}
Now, sum up \eqref{deps} over all $[M]\in \Omega$.
$$\sum_{[M]\in \Omega}\dim_{\C}\left(\text{End}_{\C^{\omega_{[M]}}\mathcal{I}_{\mathcal{G}}(M)}({\C^{\alpha}G}{\otimes}_{\C^{\alpha}N}M)\right)=\sum_{[M]\in \Omega}\delta_M
=\sum_{[M]\in \Omega}\frac{|G|^2\cdot \dim_{\C}(M)^2}{|N|^2\cdot|\mathcal{I}_{\mathcal{G}}(M)|}=\cdots$$
Plugging \eqref{orbstabeq} inside the above equation and then using \eqref{dimN} we get
$$\cdots=\frac{|G|}{|N|}\sum_{[M]\in \Omega} \dim_{\C}(M)^2\cdot|T_M|=|G|.$$
This completes the proof of (2).

\textbf{Proof of (3).}
The ${G/N}$-graded algebra $\bigslant{\C^{\alpha}G}{N}$ is equi-dimensional as a quotient grading of the equi-dimensional TGA grading $\C^{\alpha}G$.
By Lemma \ref{equidimquot} and \eqref{cdots}, the ${G/N}$-graded simple summands
\begin{equation}\label{yMtga}
y_M\left(\C^{\omega_{[M]}}\mathcal{I}_{\mathcal{G}}(M)\right),\ \ [M]\in \Omega
\end{equation}
of this semisimple finite-dimensional algebra are equi-dimensional.
Theorem \ref{th:equidim} can now be applied for every $[M]\in \Omega$, putting $\mathcal{A}:=\C^{\omega_{[M]}}\mathcal{I}_{\mathcal{G}}$ as the equi-dimensional algebra,
and the element responsible for the induction as $y_M\in\N [{G/N}]$.
Since the induction \eqref{yMtga} is equi-dimensional,
then by \eqref{respind}, $y_M$ can be represented as a multiple of a summation over a complete set of representatives of $[{G/N}:\mathcal{I}_{\mathcal{G}}(M)]$ in ${G/N}$.
That is
\begin{equation}\label{yMa}
y_M=a\cdot \sum_{t\in T_{M}}t
\end{equation}
for some positive integer $a$. The augmentation of the expression \eqref{yMa} is
\begin{equation}\label{yMa2}
\epsilon(y_M)=a\cdot |T_{M}|.
\end{equation}
On the other hand, by Lemma~\ref{isogr}, we have
%\eqref{Endgris} and \eqref{grungr}, noticing the dimension rule \eqref{augdim} for induced gradings, we have
$\epsilon(y_M)=\text{dim}_{\C^{\omega_{[M]}}\mathcal{I}_{\mathcal{G}}(M)}({\C^{\alpha}G}{\otimes}_{\C^{\alpha}N}M)$ and hence
\begin{equation}\label{augyM}
\epsilon(y_M)=
\frac{\dim_{\C}\left({\C^{\alpha}G}{\otimes}_{\C^{\alpha}N}M\right)}{\dim_{\C}\left(\C^{\omega_{[M]}}\mathcal{I}_{\mathcal{G}}(M)\right)}=
\frac{|G|\cdot \dim_{\C}(M)}{|N|\cdot|\mathcal{I}_{\mathcal{G}}(M)|}.
\end{equation}

Comparing equations \eqref{yMa2} and \eqref{augyM}, and applying \eqref{orbstabeq} we establish
$$a=\frac{|G|\cdot \dim_{\C}(M)}{|N|\cdot|\mathcal{I}_{\mathcal{G}}(M)|\cdot |T_{M}|}=\dim_{\C}(M).$$
This equality, together with \eqref{yMa}, yields \eqref{yM} and completes the proof of (3).
\qed

\begin{corollary}\label{transit}
    A quotient $G/N$-grading of a TGA grading class $\C^{\alpha}G$ is graded-simple if and only if the $G/N$-action \eqref{act} on Irr$(\C ^{\alpha}N)$ is transitive.
    In this case, the dimensions of the irreducible $\C^{\alpha}N$-modules are all equal and satisfy
    \begin{equation}\label{eq:GNdivN}
    |N|^2\cdot|\mathcal{I}_{\mathcal{G}}(M)|=\dim_{\C}(M)^2 \cdot|G|,\ \ \forall [M]\in\text{Irr}(\C ^{\alpha}N).
    \end{equation}
\end{corollary}
\begin{proof}
    By \eqref{transgrquotient}, graded simplicity of the quotient grading and transitivity of the $G/N$-action are both equivalent to $|\Omega|=1$.
    The equality of dimensions follows from Lemma \ref{samedim}, and \eqref{eq:GNdivN} follows from \eqref{orbstabeq} and \eqref{dimN} for $|\Omega|=1$.
\end{proof}
\section{Non-degenerate twisted group algebra grading classes and their quotients}\label{cpq}
We move to algebras of matrices $M_n(\C)$.
As earlier explained, these algebras can be endowed with TGA gradings $\C^{\alpha}G$, for non-degenerate $[\alpha]\in H^2(G,\C^*)$ over a CT group $G$ of order $n^2$.
Imposing the non-degeneracy condition on $[\alpha]\in H^2(G,\C^*)$ in \eqref{transgrquotient} for $N\lhd G$,
we get that under the action \eqref{act} there is one $G/N$-orbit of Irr$(\C ^{\alpha}N)$ ( i.e. $|\Omega|=1$).
The fine part of this $G/N$-grading, that is $\C^{\omega_{[M]}}\mathcal{I}_{\mathcal{G}}(M)$ is a matrix algebra as well.
We obtain the following consequence, which can also be deduced directly from the classical Mackey's correspondence \cite[Theorem 8.3]{M58}.
\begin{corollary}\label{elqu}
    Let $[\alpha]\in H^2(G,\C^*)$ be non-degenerate, let $N\lhd G$ be any normal subgroup, and let $[M]\in$Irr$(\C ^{\alpha}N)$.
    Then with the above notation, the obstruction class $\omega_{\mathcal{G}}([M])\in H^2(\mathcal{I}_{\mathcal{G}}(M),\C^*)$ is non-degenerate.
    In particular, the inertia group $\mathcal{I}_{\mathcal{G}}(M)$ is of CT.
\end{corollary}
Notice that if both $[\alpha]\in H^2(G,\C^*)$ and its restriction to $N\lhd G$ are non-degenerate, then the unique element $[M]\in$Irr$(\C ^{\alpha}N)$ is stabilized by the entire quotient.
By Corollary \ref{elqu}, $G/N$ is of CT which proves Corollary B.

Any grading of $M_n(\C)$, in particular any quotient grading of non-degenerate TGA gradings, is obviously graded-simple.
Using Lemma \ref{samedim} and Corollary \ref{transit} for non-degenerate $[\alpha]\in H^2(G,\C^*)$ we deduce

\begin{lemma}\label{abelag}\cite[Lemma 2.2]{david2013isotropy}
Let $[\alpha]\in H^2(G,\C^*)$ be a non-degenerate cohomology class and let $N\lhd G$.
Then the $G/N$-action \eqref{act} on Irr$(\C ^{\alpha}N)$ is transitive. In particular, all the irreducible $\C^{\alpha}N$-modules are of the same dimension.
If $N\lhd G$ is further assumed to be isotropic with respect to $[\alpha]$, then all the irreducible $\C^{\alpha}N$-modules are 1-dimensional,
that is $N$ is abelian.
\end{lemma}
What about elementary quotients? By \eqref{transgrquotient} for non-degenerate $[\alpha]\in H^2(G,\C^*)$,
a quotient $G/N$-grading $[\bigslant{\C^{\alpha}G}{N}]$ is elementary if and only if it admits no stabilizers under \eqref{act}. This fact, along with Lemma \ref{abelag}, yields
\begin{lemma}\label{elemetaryquotient}
    Let $[\alpha]\in H^2(G,\C^*)$ be non-degenerate.
    Then a quotient $G/N$-grading class $[\bigslant{\C^{\alpha}G}{N}]$ is elementary if and only if $G/N$ acts freely on Irr$(\C ^{\alpha}N)$ via \eqref{act}.
In this case, since the action is also transitive, we have $|$Irr$(\C ^{\alpha}N)|=|G/N|$, and
    all the irreducible $\C^{\alpha}N$-modules are of the same dimension given by
    \begin{equation}\label{GNdivN}
    \dim_{\C} (M)^2=\frac{|N|}{|\text{Irr} (\C ^{\alpha}N)|}=\frac{|N|^2}{|G|}, \ \ \forall[M] \in \text{Irr}(\C ^{\alpha}N).
    \end{equation}
\end{lemma}
\begin{corollary}\label{cor:squarefreeelem1}
    Let $[\mathcal{G}]:=[\C^{\alpha}G]$ be a TGA grading class of a complex matrix algebra, that is $[\alpha]\in  H^2(G,\C ^*)$ is non-degenerate, let $N\lhd G$ and let $[M]\in$Irr$(\C ^{\alpha}N)$.
    Suppose that $|G|$ is cube-free.
    Then the quotient grading class $[\bigslant{\mathcal{G}}{N}]$ is elementary if and only if $|G/N|$ is square-free.
\end{corollary}
\begin{proof}
    Assume that the quotient grading class $[\bigslant{\mathcal{G}}{N}]$ of $[\mathcal{G}]$ is elementary. Then by~\eqref{GNdivN}, $|G/N|$ divides
    $|N|$. Therefore, since $|G|=|N|\cdot|G/N|$ is cube-free then $|G/N|$ is
    square-free. Conversely, suppose that $|G/N|$ is square-free. The quotient grading class $\left[\bigslant{\mathcal{G}}{N}\right]$ is induced from $\C^{\omega([M])}
    \mathcal{I}_{\mathcal{G}}(M)$, where by Lemma~\ref{elqu} the inertia $\mathcal{I}_{\mathcal{G}}(M)<G/N$ is of CT.
    Since a group of CT is of square order and since the order of $G/N$ is square-free, its only subgroup of CT is $\{e\}$.
    Consequently, the quotient grading class $[\bigslant{\mathcal{G}}{N}]$ is elementary.
\end{proof}
We remind that an elementary crossed product is an elementary grading of the matrix algebra whose character is the sum of all the elements in the grading group.
By \eqref{augdim}, the dimension of such elementary crossed product is obtained by squaring the order of the grading group.
Therefore, if an elementary crossed product is a quotient $G/N$-grading of a TGA grading $\C^{\alpha}G$ of dimension $|G|$, then comparing dimension we deduce that
\begin{equation}\label{NGN}
|N|=|G/N|\left(=\sqrt{|G|}\right).
\end{equation}
\begin{lemma}\label{alemma}
Let $N\lhd G$ and let $[\alpha]\in H^2(G,\C^*)$ be non-degenerate. Suppose that the quotient $\left[\bigslant{\C^{\alpha}G}{N}\right]$ is elementary.
Then this quotient is an elementary crossed product grading class if and only if \eqref{NGN} holds.
\end{lemma}
Relying on Lemmas \ref{elemetaryquotient} and \ref{alemma}, the following result can be considered as a primal answer to Question \ref{questiona}.
\begin{lemma}\label{ECPquot}
    Let $[\alpha]\in H^2(G,\C^*)$ be non-degenerate.
    Then the quotient $[\bigslant{\C^{\alpha}G}{N}]$ is an elementary crossed product grading class if and only if
    \begin{enumerate}
\item   $G/N$ acts freely on Irr$(\C ^{\alpha}N)$ via \eqref{act}, and
\item $|N|=|G/N|$.
\end{enumerate}
\end{lemma}
Since the action \eqref{act} is also transitive under the non-degeneracy assumption (Lemma \ref{abelag}),
(1) particularly yields that $|G/N|=$Irr$(\C ^{\alpha}N)|$. Thus, (2) can be replaced in Lemma \ref{ECPquot} with the condition
$$(2')\ \ \ \  |N|=|\text{Irr}(\C ^{\alpha}N)|,$$
which occurs exactly when all the irreducible $\C^{\alpha}N$-modules are 1-dimensional, that is if and only if $N$ is abelian and the restriction of the cohomology class $[\alpha]$
to $N$ is trivial (see \cite[Lemma 1.2(i)]{Higgs88}).

Notice that if $[\bigslant{\C^{\alpha}G}{N}]$ is an elementary grading then by~\eqref{GNdivN}, $|G|$ is a divisor of $|N|^2$.
An extreme case is an elementary crossed product in which, by Lemma~\ref{alemma}, $|G|=|N|^2$. Using the notation \eqref{poset} we conclude
\begin{corollary}\label{cor:CPmax}
Any elementary crossed product quotient grading class of a TGA grading class $\C^{\alpha}G$ of $M_n(\C)$ is a maximal elementary quotient grading class of this TGA grading class.
\end{corollary}
\subsection{Lagrangians}\label{Lagrangians}
We borrow the bilinear form notation to the group-theoretic forms determined by 2-cocycles as in \cite{david2013isotropy}.
\begin{definition}\label{lagrangiandef} \cite{david2013isotropy}
    A subgroup $N< G$ is a \textit{Lagrangian} with respect to the non-degenerate cohomology class $[\alpha]\in H^2(G,\C^*)$ if
    \begin{enumerate}
        \item $N$ is isotropic with respect to $[\alpha]$, that is res$^G_N([\alpha])=[1]$, and
        \item $|N|=|G/N|=\sqrt{|G|}$.
    \end{enumerate}
    Alternatively, by, \cite{Schur1904} (see also \cite[Lemma 1.2(i)]{Higgs88}), condition (1) in Definition \ref{lagrangiandef} can be replaced by

    $(1')$ There is a 1-dimensional irreducible $\C^{\alpha}N$-module.
\end{definition}
If $[\alpha]\in H^2(A,\C^*)$ is a non-degenerate class of an abelian group $A$, then with the notation of Theorem \ref{CTab} it is not hard to verify that both $A_1$ and $\phi(A_1)$
are Lagrangians with respect to $[\alpha]$. The converse, however, does not always hold. If $A$ is not elementary abelian then $[\alpha]$ admits a Lagrangian $L<A$ without any complement subgroup,
see \S\ref{Hgradings}.

The following result fits \cite[Main Theorem]{h88} to nilpotent groups.
\begin{lemma}\label{isotropicabel}
    Let $[\alpha]\in H^2(N,\C^*)$, where $N$ is a finite nilpotent group. Then there exists an $\alpha$-isotropic subgroup $H<N$ with
    \begin{equation}\label{minind}
    [N:H]=\min\left\{\dim_{\C}(M)\mid[M]\in\Irr(\C ^{\alpha}N)\right\}.
    \end{equation}
    In particular, if $[\alpha]$ is non-degenerate, then it admits a Lagrangian \cite[Prop. 1.7]{david2013isotropy}.
\end{lemma}
\begin{proof}
    Since $N$ is nilpotent, it is a direct product $N=\prod_pN_p$ of its $p$-Sylow subgroups.
For every prime $p$, let $H_p<N_p$ be $\alpha$-isotropic of minimal index in the $p$-Sylow subgroup $N_p$.
    The direct product $H:=\prod_pH_p$ is a subgroup of $N$ of index
    \begin{equation}\label{indices}
    [N:H]=\prod_p[N_p:H_p].
    \end{equation}
    Moreover, since the subgroups $H_p$ are $\alpha$-isotropic then so is their direct product $H$.
By \cite[Main Theorem]{h88} the r.h.s. of \eqref{indices} is equal to the greatest common divisor of the dimensions of the irreducible $\C^{\alpha}N$-modules.
For a finite nilpotent group this g.c.d. is the minimum of such dimensions (see \cite[\S 3.7]{karpilovsky}), and hence the $\alpha$-isotropic  subgroup $H<N$ satisfies \eqref{minind}.

The second part of the claim follows from the fact that
    if $[\alpha]$ is non-degenerate, then the unique element $[M]\in\Irr(\C ^{\alpha}N)$ satisfies
    $\dim_{\C}(M)=\sqrt{|N|}$. Therefore, the $\alpha$-isotropic subgroup $H<N$ as above is of index $\sqrt{|N|}$ and so is a Lagrangian.
\end{proof}
Even though non-degenerate cohomology classes over nilpotent groups of CT admit Lagrangians, one cannot expect that these are normal \cite[\S 1]{david2013isotropy}. From Lemma \ref{isotropicabel} we establish
\begin{corollary}\label{lagrangian}
    Let $[\alpha] \in H^2(G,\C^*)$ be non-degenerate and let $N\lhd G$ be a normal nilpotent subgroup
    such that the quotient $G/N$-grading class $[\bigslant{\C^{\alpha}G}{N}]$ is elementary. Then $N$ contains a Lagrangian $H<G$ with respect to $[\alpha]$.
\end{corollary}
\begin{proof}
Since $[\bigslant{\C^{\alpha}G}{N}]$ is elementary, then by Lemma \ref{elemetaryquotient} all the irreducible $\C^{\alpha}N$-modules are of the same dimension,
namely $\frac{|N|}{\sqrt{|G|}}$ (see \eqref{GNdivN}).
In particular,
    \begin{equation}\label{sizeq}
    \min\{\dim_{\C}(M)\mid[M]\in\Irr(\C ^{\alpha}N)\}=\frac{|N|}{\sqrt{|G|}}.
    \end{equation}
    By Lemma \ref{isotropicabel}, there exists an $\alpha$-isotropic subgroup $H<N$ whose index is equal to the l.h.s. of \eqref{sizeq}, thus $$[N:H]=\frac{|N|}{\sqrt{|G|}}.$$
    We deduce that $|H|=\sqrt{|G|}$ and hence $H$ is a Lagrangian.
\end{proof}
\subsection{Proof of Theorem C}\label{Proof of Theorem C}
Suppose first that the normal subgroup $N\lhd G$ is a Lagrangian with respect to the non-degenerate cohomology class $[\alpha]\in H^2(G,\C^*)$.
In particular, by Condition (2) in Definition \ref{lagrangiandef}, equation \eqref{NGN} holds, that is $|N|=|G/N|$.
Lemma \ref{ECPquot} tells us then that in order to prove that $[\bigslant{\C^{\alpha}G}{N}]$ is an elementary crossed product class,
it is enough to show that the $G/N$-action \eqref{act} on $\Irr(\C ^{\alpha}N)$
is free. Indeed, owing to Lemma \ref{abelag}, all the irreducible $\C^{\alpha}N$-modules are 1-dimensional.
By \eqref{eq:GNdivN} and \eqref{NGN} we deduce that the corresponding stabilizers
$\mathcal{I}_{\mathcal{G}}(M)$ are trivial since
\begin{equation}\label{eqeq}
|\mathcal{I}_{\mathcal{G}}(M)|=\frac{|G|}{|N|^2}=1,\ \ \forall [M]\in\Irr(\C ^{\alpha}N).
\end{equation}
Consequently, $G/N$ acts freely on $\Irr(\C ^{\alpha}N)$ via \eqref{act}, proving the ``if" direction.

Conversely, suppose that $[\bigslant{\C^{\alpha}G}{N}]$ is an elementary crossed product grading class.
Then by Lemma \ref{ECPquot}, $|N|=|G/N|=\sqrt{|G|}$, and all the irreducible $\C^{\alpha}N$-modules are 1-dimensional.
Then $N\lhd G$ fulfills the conditions $(1')$ and (2) of Definition \ref{lagrangiandef} for being a Lagrangian with respect to $[\alpha],$ settling the ``only if" direction of the theorem.
\qed

\subsection{Proof of Theorem F}\label{Proof of Theorem F}
Assume that $H$ is IYB. Then, by definition it affords a bijective 1-cocycle $\pi\in Z^1(H,A)$, where $A$ is an $H$-module. In particular, $|H|=|A|$.
Endow $\check{A}:=$Hom$(A,\C^*)$ with the corresponding diagonal $H$-action as follows. For every $h\in H$ and $\chi\in\check{A}$, the character $h(\chi)\in\check{A}$ is defined using the pairing $\langle\cdot,\cdot\rangle$ between $\check{A}$ and $A$:
$$\langle h(\chi),a\rangle:=\langle\chi,h^{-1}(a)\rangle,\ \ \forall a\in A.$$
Then one obtains a semidirect product $G:=\check{A}\rtimes H$ satisfying
\begin{equation}\label{GAH}
|G|=|\check{A}|\cdot |H|=|\check{A}|^2.
\end{equation}
As shown in \cite[\S 8]{eg3}, there exists a non-degenerate class $[\alpha]\in H^2(G,\C^*)$, such that
\begin{equation}\label{resGH}
\text{res}^G_{\check{A}}([\alpha])=[1].
\end{equation}
Equations \eqref{resGH} and \eqref{GAH} are just conditions (1) and (2) in Definition \ref{lagrangiandef} respectively (with $N:=\check{A}$),
suggesting that $\check{A}$ is a Lagrangian with respect to $[\alpha]$.
Theorem C, just proven, fits into this setup saying that $[\bigslant{\C^{\alpha}G}{\check{A}}]$ is an elementary crossed product grading class of $G/\check{A}\cong H$, completing the ``only if" direction.

For the ``if" direction,
suppose that the elementary crossed product grading class corresponding to $H$ is a quotient of a TGA grading, that is
$$\left[\sum _{h\in H}h(\mathbb{C})\right]=\left[\bigslant{\C^{\alpha}G}{N}\right]$$
for some non-degenerate $[\alpha]\in H^2(G,\C^*)$, where $H \cong G/N$. Theorem C can then be used here again, now in the opposite direction, saying that $N$ is a
Lagrangian with respect to $[\alpha]$, in particular $N$ is abelian by Lemma \ref{abelag} and $|N|=|H|$.
By \cite[Theorem 4.6]{bg} there exists a bijective 1-cocycle $\pi_{\alpha}\in Z^1(H,\check{N})$, where $\check{N}:=$Hom$(N,\C^*)$.
By the definition $H$ is IYB.
\qed

\subsection{Proof of Theorem D}\label{Hgradings}
First, the ``if" part of (1) is Corollary~\ref{cor:CPmax} (even without the assumption that $A$ is abelian).

Assume now that for $N\lhd A$, the quotient $A/N$-grading class $[\bigslant{\C^{\alpha}A}{N}]$ is elementary.
Then, since $N$ is abelian, and in particular normal nilpotent subgroup of A, Corollary~\ref{lagrangian} assures that $N$ contains a Lagrangian $L$ with respect to $[\alpha]$.
Since $A$ is abelian, $L$ is normal in $A$, so the class $[\C^{\alpha}A]$ admits an $A/L$-quotient. Since $L<N$ we obtain
\begin{equation}\label{modNL}
\left[\bigslant{\C^{\alpha}A}{N}\right]\leq\left[\bigslant{\C^{\alpha}A}{L}\right].
\end{equation}
By Theorem C, the r.h.s. of \eqref{modNL} is an elementary crossed product class as a quotient of $[{\C^{\alpha}A}]$ by this Lagrangian $L$ with respect to $[\alpha]$.
Consequently, if the l.h.s. of \eqref{modNL} is maximal among the elementary quotients of $[{\C^{\alpha}A}]$ then $L=N$, i.e. $[\bigslant{\C^{\alpha}A}{N}]$ is itself an elementary crossed product.
The ``only if" direction of (1) is proven.

We now prove (2).
By the first part of the theorem, any maximal elementary quotient $A/N$-grading class of $[{\C^{\alpha}A}]$ is an elementary crossed product,
and so Theorem C asserts that $N$ is a Lagrangian with respect to
$[\alpha]$. In particular, $|N|=n$. Suppose that $A$ is elementary abelian of order $n^2$.
Certainly, $A$ is a direct product $C_q^{2r}$ of $2r$-copies of a cyclic group of square-free order $q$, where $q^r=n$.
In this case both the Lagrangian $N<G$ and the corresponding quotient group $G/N$ are isomorphic to $C_q^{r}$,
and hence the quotient class $[\bigslant{\C^{\alpha}A}{N}]=\left[\sum_{\gamma\in C_q^{r}}\gamma(\C)\right]$ is uniquely determined. This settles the ``if" direction of (2).

We prove the ``only if" part of (2) by negation. Assume that the abelian group $A$ of CT is not elementary abelian. For simplicity, start from the rank=2 case, that is
$A=\langle x \rangle \times \langle y\rangle\cong C_{n}\times C_{n}$. By the negation assumption, the exponent $n$ is not square-free and hence can be presented as $n=l^2\cdot m$
for $l,m\in\N$ with $l\geq 2$.
Let $\{u_{\sigma}\}_{\sigma\in A}$ be a basis of invertible homogeneous elements of $\C^{\alpha}A$. By Theorem \ref{CTab} and equation \eqref{commust}
$$\left[u_{x^{l}\cdot y^{l'}},u_{x^{m}\cdot y^{m'}}\right]=\chi_{\alpha}\left(x^{l}\cdot y^{l'},x^{m}\cdot y^{m'}\right)=\zeta^{l\cdot m'-l'\cdot m},$$
where $\zeta$ is a root of unity of order $n$.
By a simple calculation $[u_{x^j},u_{y^k}]=\zeta^{jk}$ for every $j,k$. Therefore, the subgroups
$$L_1:=\langle x\rangle\cong C_n,\text{ and }L_2:=\left\langle x^{l\cdot m},y^{l}\right\rangle\cong C_l\times C_{l\cdot m}  $$
are Lagrangians with respect to $[\alpha]$ with
$$A/L_1:\cong C_n,\text{ and }A/L_2\cong C_l\times C_{l\cdot m}.  $$
Since $l$ is at least 2, the quotient group $C_l\times C_{l\cdot m}  $ in the right isomorphism is not cyclic, and therefore $A/L_1$ and $A/L_2$
are not isomorphic. In turn, the corresponding quotient grading classes $[\bigslant{\C^{\alpha}A}{L_1}]$ and  $[\bigslant{\C^{\alpha}A}{L_2}],$
which are maximal elementary by Theorems C and D(1), are distinct. This contradicts the uniqueness assumption and so proves (2) for abelian CT groups of rank=2.

Back to the more general case, let $A=A_1\times \phi(A_1)$ be any decomposition corresponding to Theorem \ref{CTab}.
Again by the assumption $A$ is not elementary abelian, and so the order of at least one of the cyclic subgroups $\langle x_j\rangle<A_1$ (see \eqref{A1dec}) is not square-free.
Letting $B:=\{1,\cdots,r\}\setminus \{ j\}$, then by Corollary \ref{CorCTab}, the restriction of $[\alpha]$ to both subgroups
$$H:=\langle x_j\rangle\times \phi (\langle  x_j\rangle)<A,\text{  and }\tilde{H}:=\langle\{ x_i\}_{i\in B}\rangle\times \phi \left(\langle\{ x_i\}_{i\in B}\rangle\right)<A$$
is also non-degenerate. Note that $H\times \tilde{H}=A$.
Next, let $\tilde{L}<\tilde{H}$ be a Lagrangian with respect to the non-degenerate restriction of $[\alpha]$ to $\tilde{H}$.
For the abelian CT group $H$ of rank=2 we have just proven that
there exist Lagrangian subgroups $L_1,L_2<H$ with respect to the restriction of $[\alpha]$ to $H$,
such that their corresponding quotient groups $H/L_1$ and $H/L_2$ are non-isomorphic.
As can easily be checked, both subgroups
$$L_1\times \tilde{L}<H\times \tilde{H},\text{ and } \ L_2\times \tilde{L}<H\times \tilde{H}=A$$ are Lagrangians with respect to $[\alpha]$.
The corresponding quotient groups
$$A/(L_1\times \tilde{L})\cong (H/L_1)\times(\tilde{H}/\tilde{L}),\text{ and }A/(L_2\times \tilde{L})\cong (H/L_2)\times(\tilde{H}/\tilde{L})$$ are not isomorphic.
In turn, the corresponding quotient grading classes $$\left[\bigslant{\C^{\alpha}A}{L_1\times \tilde{L}}\right], \text{ and }\left[\bigslant{\C^{\alpha}A}{L_2\times \tilde{L}}\right],$$
which are maximal elementary by Theorems C and D(1), are distinct.
This again contradicts the uniqueness assumption and completes the proof of Theorem D(2).
 \qed

\section{Proof of Theorem G}\label{proofofF}
As already mentioned in the introduction, quotient morphisms
determine a well-defined relation on the set of graded-equivalence
classes of a given algebra. If
an algebra is finite-dimensional, then this relation, restricted
to classes of its connected gradings, is a partial order
\cite[Proposition 2.8]{GS16}.
\begin{definition}\label{scdiagdef}
	Let $\mathcal{A}$ be a finite-dimensional $\mathbb{C}$-algebra.
	\begin{enumerate}
		\item The diagram of isomorphism types of groups that are
		associated to the connected grading classes of the algebra $\mathcal{A}$ with the
		corresponding quotient morphisms is denoted by
		$\Delta(\mathcal{A})$. \item \cite[Proposition
		2.10]{cibils2011},\cite[Definition 8 and \S 4
		]{cibils2012universal} The intrinsic fundamental group
		$\pi_1(\mathcal{A})$ of $\mathcal{A}$ is the inverse limit of the
		diagram $\Delta(\mathcal{A})$.
	\end{enumerate}
\end{definition}
In this section we compute the intrinsic fundamental group of the diagonal algebras $$\C ^j=\oplus_1^j \C, \ \ j=4,5.$$
The $j=4$ case is done in \S\ref{H(1)} and the $j=5$ case is done in \S\ref{H(2)}.

We first recall the structure of the maximal graded classes of diagonal algebras.
This was firstly done in \cite{das2008}, here we give the formulation of \cite[Theorem 4.12]{GS16}. The reader is also referred to the discussion prior to that theorem.
In order to formulate this result, we need the following.\\
1. Note that the diagonal algebra $\C ^n$ is equipped with the group algebra connected grading $\C G$ for every abelian group $G$ of order $n$. Certainly, all these gradings are maximal.\\
2. The second ingredient is the free product grading hereby defined.
\begin{definition}
    Let $\mathcal{A}=\bigoplus_{j=1}^r\mathcal{A}_j$ be a direct sum of algebras, and let
    \begin{equation}\label{algrading}
    \mathcal{G}_{j}:\mathcal{A}_j=\bigoplus _{\gamma\in \Gamma_j}(\mathcal{A}_j)_{\gamma}, \ \ j=1,\cdots,r
    \end{equation}
    be gradings of the algebras $\mathcal{A}_j$ by some groups $\Gamma_j$ with trivial elements $e_j$.
    Then the {\it free product grading} of \eqref{algrading} is the grading
    $$\coprod_{j=1}^r \mathcal{G}_j:\mathcal{A}=\bigoplus_{\gamma\in\Gamma}\mathcal{A}_{\gamma}$$
    of $\mathcal{A}$ by the free product of groups $\Gamma=\coprod_{j=1}^r \Gamma_j,$ where
    \begin{equation}\label{coprodgrade}
        \mathcal{A}_{\gamma}=\left\{
    \begin{array}{rl}
    \bigoplus_{j=1}^r(\mathcal{A}_j)_{e_j},& \text{ if } \gamma=e\\
    (\mathcal{A}_j)_{\gamma},&  \text{ if }\gamma\in \Gamma_j\setminus {e_j},\\
    0,& \text{ otherwise}.
    \end{array}\right.
    \end{equation}
\end{definition}
Then, if the $\Gamma_j$-gradings $\mathcal{G}_j$ of $\mathcal{A}_j$ are connected for all $j=1,\cdots,r$,
then so is the $\coprod_{j=1}^r \Gamma_j$-grading $\coprod_{j=1}^r \mathcal{G}_j$ of $\mathcal{A}$ (see \cite[Lemma 6.4]{cibils2010}).
\begin{theorem}\label{da}(\cite[Theorem 5]{das2008}, see also \cite[Theorem 4.12]{GS16})
    The maximal connected grading classes of the diagonal algebra $\C^n$ are free product classes $\left[\coprod_{j=1}^r\C G_j\right]$ of ordinary group algebras,
    where $G_j$ are abelian groups with $\sum_{j=1}^r|G_j|=n,$ such that no more than one group $G_j$ is trivial.
\end{theorem}
In the proof of Theorem G we also use the following straightforward lemma.
\begin{lemma}\label{lemma:trivcom}
Let $\mathcal{A}$ be a finite-dimensional semi-simple algebra whose maximal grading classes are $\{[\mathcal{G}_j]\}_{1\leq j\leq r}$,
graded by the corresponding groups $\{\Gamma_j\}_{1\leq j\leq r}$.
Suppose that for all $1\leq i\leq r-1$, the only common quotient of $[\mathcal{G}_r]$ with $[\mathcal{G}_i]$ is the trivially graded class of $\mathcal{A}$.
Then $\pi _1(\mathcal{A})$ is a direct product
    of $\Gamma_r$ with the pull-back of the full sub-diagram of $\Gamma_1,\Gamma_2,\ldots ,\Gamma_{r-1}$.
\end{lemma}
In \cite{cibils2010} the intrinsic fundamental groups of $\C^2,\C^3,\C^4$ are given as follows.
\begin{itemize}
    \item  By Theorem~\ref{da} for $\C ^2$, there is a unique maximal grading class graded by $C_2$ and hence (see \cite[Proposition 6.3]{cibils2010}) $$\pi _1(\C ^2)\cong C_2.$$
    \item Similarly, for $\C ^3$ there are two maximal grading classes by $C_3$ and by $C_2$. The common quotient of these classes consists only of the trivially graded class.
By Lemma \ref{lemma:trivcom} (see \cite[Corollary 6.8]{cibils2010}) $$\pi _1(\C ^3)\cong C_3\times C_2.$$
    \item The four maximal grading classes of $\C ^4$ (see $(1)$-$(4)$ in \S\ref{H(1)} herein) are correctly listed also in \cite[Theorem 6.9]{cibils2010}.
    However, it is mistakenly concluded
    that $\pi_1 (\C ^4)$ is the direct product of the corresponding groups (as in the case of $\C ^2$ and $\C ^3$), not taking into account the common quotients between these maximal classes.
    We wish to recalculate $\pi_1 (\C ^4)$.
\end{itemize}

We note that in the proof of both parts of Theorem G we use $e$ as the identity element of different groups.
\subsection{Proof of Theorem G(1)}\label{H(1)}
    By Theorem~\ref{da}, ${\C ^4}$ admits $4$ maximal graded equivalence classes.
    \begin{enumerate}
        \item The group algebra class $[\C C_3]$, attained by the isomorphism $\C^4\cong\C C_3\oplus \C ,$ and graded by the cyclic group $C_3$.
        \item The group algebra class $[\C C_4]$ graded by the cyclic group $C_4=\langle x  \rangle$.
        \item The group algebra class $[\C (C_2\times C_2)]$ graded by the Klein group $C_2\times C_2={\langle \sigma  \rangle \times \langle \tau  \rangle}$.
        \item The free product class $[\C C_2* \C C_2]$ determined by the isomorphism $\C^4\cong\C C_2\oplus \C C_2$, and graded by the free product $C_2*C_2=\langle a \rangle * \langle b\rangle.$
    \end{enumerate}
    Here is a description of the common quotients between the above gradings. Clearly, the only common quotient of the grading class (1)
    with the other grading classes is the trivially graded class.
    Denoting the pull-back of the full sub-diagram \eqref{subdiag4} of the gradings (2),(3) and (4) by $G_4$, we obtain by Lemma~\ref{lemma:trivcom}
    \begin{equation}\label{dirsum}
    \pi_1({\C ^4})\cong C_3\times G_4.
    \end{equation}
    The other maximal grading classes (2),(3) and (4) admit the following maximal common quotient class graded by the cyclic group $C_2=\langle y \rangle$.
    \begin{itemize}
        \item The grading class determined by the isomorphism $\C^4\cong\C C_2\oplus \C C_2,$  where both copies of the group algebra $\C C_2$ are graded by $\{e,y\}$.
    \end{itemize}
    The corresponding sub-diagram of $\Delta(\C ^4)$ is
    \begin{equation}\label{subdiag4}
    \begin{tikzpicture}
    \tikzset{thick arc/.style={->, black, fill=none,  >=stealth,
            text=black}} \tikzset{node distance=2cm, auto} \node (C_4)
    {$C_4$}; \node (direct) [right of=C_4] {$C_2\times C_2$}; \node
    (free) [right of=direct] {$C_2* C_2$}; \tikzset{node distance=2cm,
        auto} \node (C_2) [below of=direct] {$C_2$}; \tikzset{node
        distance=3cm, auto} \draw[thick arc, draw=blue, dashed] (C_4) to
    node [near start] [left] {$\psi _1$} (C_2); \draw[thick arc,
    draw=blue, dashed] (direct) to node [near start] [left] {$\psi
        _2$} (C_2); \draw[thick arc, draw=blue, dashed] (free) to node
    [near start] [right] {$\psi _3$} (C_2);
    \end{tikzpicture}
    \end{equation}
    where
    \begin{equation}\label{psis}
    \psi _1:x\mapsto y, \quad \psi _2:\sigma,\tau \mapsto y,\quad \psi _3:a,b \mapsto y.
    \end{equation}
    The pull-back $G_4$ of \eqref{subdiag4} is the subgroup of the direct product $(C_4)\times (C_2\times C_2)\times (C_2*C_2)$ of \textit{admissible} triples with respect to \eqref{subdiag4}. More precisely,
    $$G_4=\left\{(g_1,g_2,g_3)\in (C_4)\times (C_2\times C_2)\times (C_2*C_2)\mid\psi_1(g_1)=\psi_2(g_2)=\psi_3(g_3)\right\}.$$
    Notice that $$z_1:=(x,\sigma,a),z_2:=(x,\sigma,b),z_3:=(x^2,\sigma \tau,e)\in G_4.$$
    We claim that $G_4=\langle z_1,z_2,z_3\rangle$.
    Indeed, let $g=(g_1,g_2,g_3)\in G_4$. Write $g_3\in C_2*C_2=\langle a \rangle * \langle b\rangle$ as a reduced word $w(a,b)$ on the letters $a,b$.
    Then the word $w(z_1,z_2)=(\tilde{g_1},\tilde{g_2},g_3)$ for some $\tilde{g_1}\in C_4$ and $\tilde{g_2}\in C_2\times C_2$.
    Therefore, $$g^{-1}\cdot w(z_1,z_2)=\left(g_1^{-1}\cdot\tilde{g_1},g_2^{-1}\cdot\tilde{g_2},e\right)\in G_4.$$
    Since this is an admissible triple and since $\psi _3(e)=e$, then
    $$g_2^{-1}\cdot\tilde{g_2}\in\ker(\psi _2)=\{e,\sigma\tau\}.$$
    \begin{itemize}
 \item   In case $g_2^{-1}\cdot\tilde{g_2}=e$ we have that
    $$g^{-1}\cdot w(z_1,z_2)=\left(g_1^{-1}\cdot\tilde{g_1},e,e\right).$$
  \item  In case $g_2^{-1}\cdot\tilde{g_2}=\sigma \tau$ we have that
    $$g^{-1}\cdot w(z_1,z_2)\cdot z_3=\left(g_1^{-1}\cdot\tilde{g_1}\cdot x^2,e,e\right).$$
    \end{itemize}
    As above, $$ g_1^{-1}\cdot\tilde{g_1}\in\ker(\psi_1)=\{e,x^2\}.$$ In both cases, multiplying by $z_1^2=(x^2,e,e)$ if needed,
    either one of the following holds
    $$g=w(z_1,z_2),\quad g=w(z_1,z_2)\cdot z_3,\quad g=w(z_1,z_2)\cdot z_1^2,\quad g=w(z_1,z_2)\cdot z_3\cdot z_1^2.$$
    This proves that $g\in \langle z_1,z_2,z_3\rangle$ and hence these three elements generate $G_4$.

    Note that $z_3$ is central of order 2, and that $\langle z_1,z_2\rangle \cap \langle z_3\rangle=\{e\}.$
    We obtain
    \begin{equation}\label{G4}
    G_4=\langle z_1,z_2\rangle \times \langle z_3\rangle\cong \langle z_1,z_2\rangle\times C_2.
    \end{equation}
    It is left to describe the group $H_4:=\langle z_1,z_2\rangle$. It is not hard to verify that
    $$\beta_4 :
    \begin{array}{c}H_4\rightarrow C_2*C_2=\langle a,b\rangle\\
    z_1\mapsto a,\ \ \ \ z_2\mapsto b
    \end{array}$$
    is a well-defined surjective homomorphism with ker$(\beta_4)=\langle (x^2,e,e)\rangle$ central of\\ order 2.
    Then $\beta_4$ gives rise to a central extension:
\begin{equation}\label{H4ext}
\quad 1\rightarrow \stackrel{\langle (x^2,e,e)\rangle}{C_2}\rightarrow H_4 \stackrel{\beta_4}{\rightarrow} \stackrel{\langle a\rangle}{C_2}*\stackrel{\langle b\rangle}{C_2}\rightarrow 1
\end{equation}
    determined by $\bar{a}^2=\bar{b}^2=(x^2,e,e)$ for $\bar{a}:=z_1\in \beta_4 ^{-1}(a)$ and $\bar{b}:=z_2\in \beta_4 ^{-1}(b)$.
    Gathering equations \eqref{dirsum}, \eqref{G4} and \eqref{H4ext} concludes the proof.\qed
\subsection{Proof of Theorem G(2)}\label{H(2)}
    By Theorem~\ref{da}, $\C ^5$ admits $5$ maximal graded equivalence classes.
    \begin{enumerate}
        \item The group algebra class $[\C C_5]$ graded by the cyclic group $C_5$.
        \item The group algebra class $[\C C_4]$, attained by the isomorphism $\C^5\cong\C C_4\oplus \C $, and graded by the cyclic group $C_4=\langle x  \rangle$.
        \item The group algebra class $[\C (C_2\times C_2)]$, attained by the isomorphism $\C^5\cong\C (C_2\times C_2)\oplus \C $,
        and graded by the Klein group $C_2\times C_2={\langle \sigma  \rangle \times \langle \tau  \rangle}$.
        \item The free product class $[\C C_2* \C C_2]$ attained by the isomorphism $\C^5\cong\C C_2\oplus \C C_2\oplus\C $, and graded by $C_2*C_2=\langle a \rangle * \langle b\rangle.$
        \item The free product class $[\C C_3* \C C_2]$, attained by the isomorphism $\C^5\cong\C C_3\oplus \C C_2$, and graded by the free product $C_3*C_2=\langle g \rangle * \langle h\rangle.$
    \end{enumerate}
    As above, here are the maximal common quotients between the above maximal grading classes.
    Clearly, the only common quotient of the grading class (1) with the other grading classes is the trivially graded class.
    Denoting the pull-back of the full sub-diagram \eqref{subdiag5} of the gradings (2),(3),(4) and (5) by $G_5$, we obtain by Lemma~\ref{lemma:trivcom}
    \begin{equation}\label{dirsum5}
    \pi_1({\C ^5})\cong C_5\times G_5.
    \end{equation}
    As in the proof of Theorem G(1), the maximal grading classes (2),(3) and (4) admit the following maximal common quotient class.
    \begin{itemize}
        \item The grading class determined by the isomorphism $\C^5\cong\C C_2\oplus \C C_2\oplus \C ,$  where both copies of the group algebra $\C C_2$ are graded by $C_2=\{e,y\}$.
    \end{itemize}
The grading class (5) admits only the trivially graded common quotient class both with the class (2) as well as with the class (3).
It remains to describe the maximal common quotient between the grading classes (4) and (5). This is
    \begin{itemize}
        \item The grading class determined by the isomorphism $\C^5\cong\C C_2\oplus \C^3 ,$  where the group algebra $\C C_2$ is graded by $C_2=\{e,\kappa\}$.
    \end{itemize}
    The corresponding sub-diagram of $\Delta(\C ^5)$ (in which we do not write trivial common quotients) is
    \begin{equation}\label{subdiag5}
    \begin{tikzpicture}
    \tikzset{thick arc/.style={->, black, fill=none,  >=stealth,
            text=black}} \tikzset{node distance=2cm, auto} \node (C_4)
    {$C_4$}; \node (direct) [right of=C_4] {$C_2\times C_2$}; \node
    (free) [right of=direct] {$C_2* C_2$}; \node (free3) [right
    of=free] {$C_3* C_2$}; \tikzset{node distance=2cm, auto} \node
    (C_2) [below of=direct] {$C_2$}; \node (C_22) [below of=free]
    {$C_2$}; \tikzset{node distance=3cm, auto} \draw[thick arc,
    draw=blue, dashed] (C_4) to node [near start] [left] {$\psi _1$}
    (C_2); \draw[thick arc, draw=blue, dashed] (direct) to node [near
    start] [left] {$\psi _2$} (C_2); \draw[thick arc, draw=blue,
    dashed] (free) to node [near start] [right] {$\psi _3$} (C_2);
    \draw[thick arc, draw=red, dashed] (free) to node [near start]
    [right] {$\phi _1$} (C_22); \draw[thick arc, draw=red, dashed]
    (free3) to node [near start] [right] {$\phi _2$} (C_22);
    \end{tikzpicture}
    \end{equation}
    where $\psi _1,\psi _2$ and $\psi _3$ are given in \eqref{psis}, and where
    \begin{eqnarray}\label{phis}
    \phi _1:
    \begin{array}{rcl}
    a &\mapsto      & \kappa \\
    b &\mapsto      &  1
    \end{array}
    ,\ \ \ \phi _2:
    \begin{array}{rcl}
    h &\mapsto      & \kappa \\
    g &\mapsto      &  1
    \end{array}.
    \end{eqnarray}
    Let $$Q_5=\{(g_1,g_2)\in(C_2*C_2)\times (C_3*C_2)\mid \phi_1 (g_1)=\phi_2 (g_2)\}$$
    be the pull-back of the red sub-diagram.
    Notice that $$u_1:=(b,e),u_2:=(e,g),u_3:=(a,h)\in Q_5.$$
    We claim that $Q_5=\langle u_1,u_2,u_3\rangle$.
    Indeed, let $(g_1,g_2)\in Q_5$. Write $g_1\in C_2*C_2$ as a reduced word $w_1(a,b)$ on the letters $a$ and $b$, and $g_2\in C_3*C_2$ as a reduced word $w_2(g,h)$ on the letters $g$ and $h$.
    Since $(g_1,g_2)$ is an admissible pair, we have
    $$\phi_1 (w_1(a,b))=\phi _2(w_2(g,h)).$$
    Denote by $n_1(a)$ and $n_1(b)$ the number of appearances of the letters $a$ and $b$ in $w_1(a,b)$ respectively.
    Similarly, denote by $n_2(h)$ and $n_2(g)$ the number of appearances of the letters $h$ and $g$ in $w_2(g,h)$ respectively.
Consider the word $w_1(u_3,u_1)$. It is not hard to check that
    \begin{equation}\label{w11}
    w_1(u_3,u_1)=\left(g_1,h^{n_1(a)}\right)=\left(g_1,h^{n_1(a)+2m}\right)\in Q_5,
    \end{equation}
for any integer $m$, which can be chosen to be large enough so as to satisfy $n_1(a)+2m\geq n_2(h)$.
    Now, plug $n_2(g)$ many copies of $u_2$ into the word \eqref{w11} in a coherent way from left to right
    to get a new word
    \begin{equation}\label{w31}w_3(u_3,u_1,u_2)=\left(g_1,g_2\cdot h^{n_1(a)+2m-n_2(h)}\right)\in Q_5
    \end{equation}
    (note that the left component is not affected by this step).
    Next, by multiplying the element \eqref{w31} from the left by the inverse of the element $(g_1,g_2)\in Q_5$ we get    $\left(e,h^{n_1(a)+2m-n_2(h)}\right)\in Q_5,$
    and therefore
    \begin{equation}\label{hen1}
    h^{n_1(a)+2m-n_2(h)}=e.
    \end{equation}
    By \eqref{w31} and \eqref{hen1} we establish
    $$(g_1,g_2)=w_3(u_3,u_1,u_2)\in \langle u_1,u_2,u_3\rangle.$$
Consequently, $u_1,u_2,u_3$ generate $Q_5$, verifying the claim.

    Next, with the above notation,
    $$\langle u_1,u_2\rangle=\langle u_1\cdot u_2\rangle=\langle (b,g)\rangle \cong C_6.$$ Also notice that there are no non-trivial relations between $(b,g)$ and $(a,h)$.
    The pull-back $Q_5$ of the red sub-diagram is then well-understood
    \begin{equation}\label{Q5frpr}
    	Q_5=\langle u_1,u_2,u_3\rangle \cong C_6* C_2.
    	\end{equation}
    This enables us to get hold of the pull-back
    $$G_5<(C_4)\times (C_2\times C_2)\times(C_2*C_2)\times (C_3*C_2)$$ of the entire diagram \eqref{subdiag5} as follows.
Firstly, as can easily be verified, the tuples
\begin{equation}\label{genG5}
\{ (x,\sigma ,a,h),(x,\sigma ,b,e),(x^2,\sigma \tau,e,e),(e,e,e,g)\}.
\end{equation}
    are admissible with respect to \eqref{subdiag5} and hence belong to $G_5$. We claim that \eqref{genG5} generate $G_5$. Indeed, let $(g_1,g_2,g_3,g_4)\in G_5$.
In particular, $g_4\in C_3*C_2$ is a word $w'(h,g)$, and the triple $(g_1,g_2,g_3)$ is admissible with respect to the ``blue" sub-diagram, and hence belong to $G_4$.
By \eqref{G4}, $(g_1,g_2,g_3)\in G_4$ is a word $w''(z_1,z_2,z_3)$. Consequently,
\begin{equation}\label{w20}
w''(x,\sigma ,a,h),(x,\sigma ,b,e),\left(x^2,\sigma \tau,e,e)\right)=\left(g_1,g_2,g_3,h^{n}\right)\in G_5
\end{equation}
for some non-negative integer $n$. Since $h$ is of finite order, the integer $n$ may be assumed to be at least the number of appearances of $h$ in the word $w'(h,g)$.
Let $n'(h)$ and $n'(g)$ be the number of appearances of $h$ and $g$ in the word $w'(h,g)$ respectively.
As before, plug $n'(g)$ many copies of $(e,e,e,g)$ into the word \eqref{w20} coherently from left to right
    to get a new word
    \begin{equation}\label{w40}w'''\left((x,\sigma ,a,h),(x,\sigma ,b,e),\left(x^2,\sigma \tau,e,e\right),(e,e,e,g)\right)=\left(g_1,g_2,g_3,g_4\cdot h^{n-n'(h)}\right)\in G_5
    \end{equation}
    (note that the left three components are not affected by this step).
Next, by multiplying the element \eqref{w40} from the left by the inverse of the element $(g_1,g_2,g_3,g_4)\in G_5$ we get
    $\left(e,e,e,h^{n-n'(h)}\right)\in G_5,$
    and therefore
    \begin{equation}\label{nhne}
    h^{n-n'(h)}=e.
    \end{equation}
    By \eqref{w40} and \eqref{nhne} we deduce that $(g_1,g_2,g_3,g_4)\in G_5$ is the word $w'''$ in \eqref{genG5}, and hence
    $$G_5=\langle ((x,\sigma ,a,h),x,\sigma ,b,e),(x^2,\sigma \tau,e,e),(e,e,e,g)\rangle.$$
    In particular,
$$(e,\sigma \tau,e,e)=(x^2,\sigma \tau,e,e)\cdot(x,\sigma ,b,e)^2$$
is central of order $2$ in $G_5$. Denote the subgroup $$H_5:=\langle (x,\sigma ,a,h),(x,\sigma ,b,e),(e,e,e,g)\rangle<G_5.$$
Then $H_5\cap \langle (e,\sigma \tau,e,e)\rangle=\{e\}$, and together with the centrality of $(e,\sigma \tau,e,e)$ in $G_5$ we have
    \begin{equation}\label{G5H5}
    G_5=H_5\times\langle (e,\sigma \tau,e,e)\rangle\cong H_5\times C_2.
    \end{equation}

    It is left to describe the group $H_5$.
As can easily be verified, the rule
    $$\beta_5 :
    \begin{array}{c}
    H_5\rightarrow Q_6\cong \stackrel{\langle u_1\cdot u_2\rangle}{C_6}*\stackrel{\langle u_3\rangle}{C_2}\\
    (x,\sigma ,b,e)\mapsto u_1,\quad (x,\sigma ,a,h)\mapsto u_3,\quad (e,e,e,g)\mapsto u_2.
    \end{array}$$
(see \eqref{Q5frpr}) determines a well-defined surjective homomorphism with
    $$\ker(\beta_5)=\langle (x^2,e,e,e)\rangle =\langle (x,\sigma ,b,e)^2\rangle =\langle (x,\sigma ,a,h)^2\rangle $$ central of order 2. Consequently, $\beta_5$ gives rise to the central extension
    \begin{equation}\label{H5ext}
\quad 1\rightarrow \stackrel{\langle (x^2,e,e,e)\rangle}{C_2}\rightarrow H_5\stackrel{\beta_5}{\rightarrow} \stackrel{\langle u_1\cdot u_2\rangle}{C_6}*\stackrel{\langle u_3\rangle}{C_2}\rightarrow 1
    \end{equation}
    determined by $$\overline{z}^6=\overline{w}^2=(x^2,e,e,e)$$ for $\overline{z}:=(x,\sigma ,b,g)\in \beta_5 ^{-1}(u_1\cdot u_2)$ and $\overline{w}:=(x,\sigma ,a,h)\in \beta_5 ^{-1}(u_3)$.
    The extension \eqref{H5ext}, together with equations \eqref{dirsum5} and \eqref{G5H5}, complete the proof.\qed

\vspace*{0.3cm}
\noindent\textbf{Acknowledgement.} We thank the referee for the important comments, which significantly improve the paper.

\end{document}